\documentclass[12pt, a4paper]{article}

\usepackage{mathptmx} 
\usepackage{amsmath}
\usepackage{amsthm}
\usepackage[colorlinks,citecolor=blue,urlcolor=blue,linkcolor=blue,linktocpage=true]{hyperref}
\usepackage{amssymb}
\usepackage{graphicx}
\usepackage[tight]{subfigure}
\usepackage{authblk}
\usepackage{bigints}

\newcommand{\st}{\mid}
\newcommand{\cond}{\,|\,} 
\newcommand{\diff}{\mathrm d}
\let\Re\relax 
\DeclareMathOperator{\Re}{Re} 

\DeclareMathOperator{\E}{E}
\newcommand{\funt}{\tau} 
\newcommand{\funl}{\lambda} 
\newcommand{\funti}{T} 
\newcommand{\funli}{\Lambda} 
\newcommand{\funta}{\varphi} 
\newcommand{\probmax}{\rho} 
\newcommand{\ras}{\sigma}
\newcommand{\len}{\ell} 
\newcommand{\leni}{\ell} 
\newcommand{\tiles}{t} 
\newcommand{\isolr}{i^+}
\newcommand{\jsolr}{j^+}
\newcommand{\isoli}{i^\ast}
\newcommand{\jsoli}{j^\ast}
\newcommand{\genfun}{q}
\newcommand{\genvar}{s}
\newcommand{\roundterm}{r}
\newcommand{\mss}{M}
\newcommand{\funcirc}{\Omega}
\newcommand{\funcircesc}{\bar \Omega}
\newcommand{\fdp}{p}
\newcommand{\lra}{\alpha}
\newcommand{\lrb}{\beta}
\newcommand{\intervx}{I}

\newtheorem{proposition}{Proposition}
\newtheorem{theorem}{Theorem}
\newtheorem{corollary}{Corollary}

\makeatletter
\renewcommand\theenumi{(\@roman\c@enumi)} 
\renewcommand\labelenumi{\theenumi} 
\makeatother

\graphicspath{{figures/}}

\begin{document}

\title{
On the number of tiles visited by a\\
line segment on a rectangular grid
}

\author[1]{Alex Arkhipov}
\author[2]{Luis Mendo}

\affil[1]{\small E-mail: \texttt{arkhipov@alum.mit.edu}}
\affil[2]{\small Universidad Polit\'ecnica de Madrid. E-mail:  \texttt{luis.mendo@upm.es}}


\maketitle

\begin{abstract}
Consider a line segment placed on a two-dimensional grid of rectangular tiles. This paper addresses the relationship between the length of the segment and the number of tiles it visits (i.e.~has intersection with). The square grid is also considered explicitly, as some of the specific problems studied are more tractable in that particular case. The segment position and orientation can be modelled as either deterministic or random. In the deterministic setting, the maximum possible number of visited tiles is characterized for a given length, and conversely, the infimum segment length needed to visit a desired number of tiles is analyzed. In the random setting, the average number of visited tiles and the probability of visiting the maximum number of tiles on a square grid are studied as a function of segment length. These questions are related to Buffon's needle problem and its extension by Laplace. 


\end{abstract}

\section{Introduction}
\label{part: intro}

Given $a, b \in \mathbb R^+$, consider a grid on $\mathbb R^2$ formed by rectangular \emph{tiles} of width $a$ and height $b$. A line segment of length $\len \in \mathbb R^+$ is located on the plane with arbitrary position and orientation. The segment is said to \emph{visit} a tile if it intersects its interior.\footnote{
The definition uses the interior of the tile, excluding the border, to avoid uninteresting results such as a ``zero-length'' segment visiting (a vertex of) $4$ tiles.}

This paper studies the relationship between the length of the segment and the number of visited tiles. The motivation comes from the classical Buffon-Laplace needle problem (i.e.~a segment with random position and orientation on a rectangular grid), of which a modified version is considered, wherein the segment position and orientation are parameters that can be chosen to maximize the number of visited tiles. In addition, the probability of visiting that maximum number of tiles in the classical (random) setting is studied.

Specifically, two different settings are considered, which correspond to the segment position and orientation being deterministic or random, respectively. In the \emph{deterministic} case, the relevant questions are:
\begin{enumerate}
\renewcommand{\labelenumi}{(1\alph{enumi})}
\renewcommand{\theenumi}{(1\alph{enumi})}
\item
\label{prob: determ dir}
What is the maximum number of tiles that the segment can visit given its length?
\item
\label{prob: determ inv}
Conversely, what length should a segment have to visit a given number of tiles?
\end{enumerate}

In the \emph{random} setting, if the segment position and orientation are uniformly distributed (this will be precisely defined later),
\begin{enumerate}
\renewcommand{\labelenumi}{(2\alph{enumi})}
\renewcommand{\theenumi}{(2\alph{enumi})}
\item
\label{prob: rand ave}
What is the average number of tiles visited by a segment of a given length?
\item
\label{prob: rand prob}
How often does the random segment visit the maximum number of tiles?
\end{enumerate}

As an example of question~\ref{prob: determ dir}, consider $a=1.35$, $b=1$. A segment of unit length can be placed as shown in Figure~\ref{fig: examples} (left) to make it visit $3$ tiles. In fact, this is the maximum number for $\len=1$. The figure also illustrates that the solution for length $2.4$ is $5$ (center), and for $4.7$ it is $8$ (right).

\begin{figure}
\centering%
\includegraphics[width=.85\textwidth]{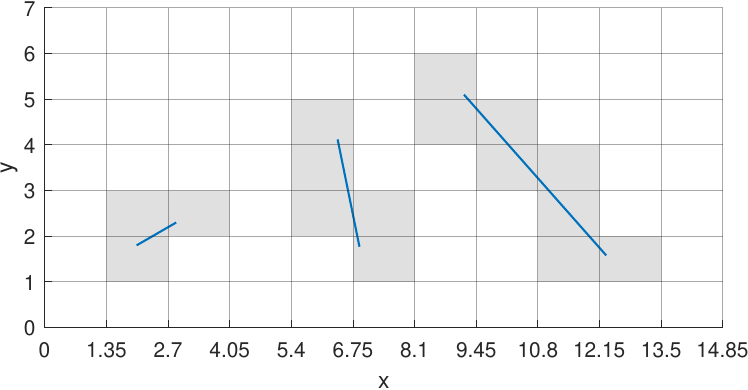}%
\caption{Examples for $a=1.35$, $b=1$; $\len=1$, $\len=2.4$ and $\len=4.7$
}%
\label{fig: examples}%
\end{figure}%

An equivalent formulation of the problem is obtained by allowing segments of length $\len$ \emph{or smaller}. The equivalence is clear from the fact that reducing the length cannot increase the number of visited tiles. Either of these formulations will be referred to as the \emph{direct} problem.

The \emph{inverse} problem~\ref{prob: determ inv} is, given $\tiles \in \mathbb N$, to determine the infimum length of all segments that visit at least $\tiles$ tiles. If the length can take any real value the infimum is not a minimum, because given any segment it can be shortened by some small amount without changing the number of visited tiles. This is a consequence of the interior of each tile being an open set.

The direct and inverse problems are closely related. Namely, if $\len$ is the infimum of all lengths that allow visiting at least $\tiles$ tiles (inverse problem), $\tiles$ is the maximum number of tiles that can be visited with lengths slightly greater than $\len$ (direct problem).

To address the remaining two questions, the notion of a \emph{random} segment of a given length needs to be precisely defined. This is done as follows. By symmetry, one endpoint of the segment can be assumed to lie in a fixed, reference tile. The position of this endpoint is \emph{uniformly} distributed on the tile. The segment orientation has a \emph{uniform} distribution on $[0,2\pi)$, and is \emph{independent} of the endpoint position. Solving the problem~\ref{prob: rand ave} of how many tiles the segment visits on average also answers, as will be seen, the inverse question (segment length to visit a given number of tiles on average). A natural, related problem~\ref{prob: rand prob} is with what probability the segment visits the maximum number of tiles. 


The questions studied in this paper are related to Buffon's needle problem and Laplace's extension of it, as stated at the outset. Buffon's original problem considers a plane with vertical lines a distance $a$ apart. A needle is placed on the plane with uniformly random position\footnote{It suffices to define position using a horizontal coordinate modulo $a$, for which a uniform distribution can be defined.} and orientation, and the probability of the needle crossing a line is studied. For a needle of length $\len = a$, this probability is $2/\pi$ (and hence repeated trials of this experiment can be used to estimate $\pi$). This is generalized in \cite{Ramaley69} to $2\len /(\pi a)$ for the expected number of crossings of a needle with arbitrary length.

The Buffon-Laplace needle problem \cite[section~1.1]{Mathai99} considers a needle randomly placed on a grid of rectangular tiles of width $a$ and height $b$. The number of visited tiles equals one plus the number of crossings almost surely. The probability of the needle staying within a single tile is computed in \cite{Arnow94} for the case where $\len < \min\{a,b\}$. One of the problems considered in this paper, as mentioned above, is the complementary question~\ref{prob: rand prob} of the probability that the needle visits the maximum number of tiles possible for its length.

The rest of the paper is organized as follows. Fundamental results are presented in \S\ref{part: fund results}, which form the basis of both the deterministic and probabilistic analyses. The direct and inverse problems for a deterministic segment are considered in \S\ref{part: max}, first for arbitrary grids and then for a square grid. The analysis for the random segment is carried out in \S\ref{part: rand}. The average number of tiles is computed for arbitrary grids, and the probability that the segment visits the maximum number of tiles is obtained for a square grid.

The symbols $\lfloor x \rfloor$ and $\lceil x \rceil$ respectively denote the floor (rounding down) and ceiling (rounding up) operations. The functions $\arcsin x$ and $\arccos x$ are defined as their principal branches in the usual way. In particular, the output angles are in $[0, \pi/2]$ for $x \in [0,1]$.

\section{Fundamentals}
\label{part: fund results}

For a grid with horizontal spacing $a$ and vertical spacing $b$, lines $x = ka$ or $y = kb$ with $k \in \mathbb Z$ will be called \emph{grid lines}. A \emph{tile} is delimited by two pairs of consecutive horizontal and vertical grid lines. The intersection points of horizontal and vertical grid lines will be called \emph{grid points}. These correspond to vertices of the tiles.

Every segment has an associated \emph{discrete bounding rectangle}, which is the minimum-size rectangle that is formed by grid lines and contains the segment. More specifically, if the segment has endpoints $(x_1,y_1)$, $(x_2,y_2) \in \mathbb R^2$, its discrete bounding rectangle has lower-left and upper-right corners respectively given as
\begin{align*}
& (\lfloor\min\{x_1, x_2\}/a\rfloor a, \lfloor\min\{y_1,y_2\}/b\rfloor b), \\
& (\lceil\max\{x_1, x_2\}/a \rceil a, \lceil\max\{y_1,y_2\}/b \rceil b).
\end{align*}
The dimensions of the discrete bounding rectangle, \emph{normalized} to the tile width and height respectively, are two integer numbers $i$, $j$. Two examples are illustrated in Figure~\ref{fig: discrete bounding rectangle and touched tiles}, both with $i=5$, $j=4$. All tiles visited by the segment are contained in the discrete bounding rectangle. Note also that the rectangle can have $i=0$ or $j=0$ if the segment coincides with part of a grid line.

\begin{figure}
\centering%
\subfigure[The segment does not pass through any interior grid points]{%
\label{fig: S_nogridpoints}%
\includegraphics[width=.49\textwidth]{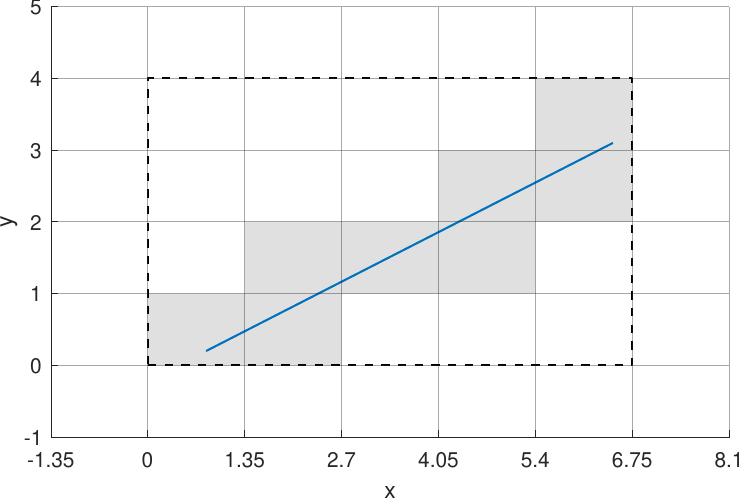}%
}\hfill%
\subfigure[The segment passes through some interior grid points]{%
\label{fig: S_gridpoints}%
\includegraphics[width=.49\textwidth]{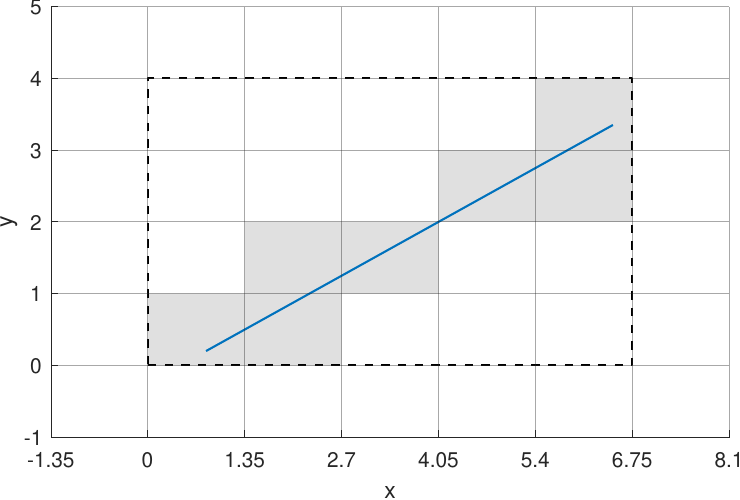}%
}%
\caption{Discrete bounding rectangle and visited tiles
}%
\label{fig: discrete bounding rectangle and touched tiles}%
\end{figure}%

\begin{proposition}
\label{prop: i+j-1}
Consider an arbitrary segment, and let $i, j$ respectively denote the normalized width and height of its discrete bounding rectangle. If $i, j \geq 1$, the number of tiles visited by the segment is at most $i+j-1$. This bound is attained if and only if the segment does not pass through any grid point in the interior of the rectangle.
\end{proposition}

\begin{proof}
The segment visits, by definition, two tiles in opposite corners of the discrete bounding rectangle. It can be assumed, without loss of generality, that those tiles are in the lower-left and upper-right corners of the rectangle, as in Figure~\ref{fig: discrete bounding rectangle and touched tiles}. The visited tiles can be thought of as following a path within the discrete bounding rectangle. Starting at the lower-left tile, the next tile can be the one to the left, the one above, or the one above and to the left. The latter case occurs if and only if the segment passes through the grid point between those two tiles.

Since the segment follows a straight line, once it ``leaves'' a row of tiles in its path from the lower-left to the upper-right corner, it can never visit any more tiles from that row. The same observation applies to the columns.

This implies that the maximum number of visited tiles is $i+j-1$, which is attained if and only if the segment avoids all grid points in the interior of the discrete bounding rectangle, as in Figure~\ref{fig: S_nogridpoints}. Note that grid points at the corners of the rectangle do not count for this; and that the segment cannot pass through any other grid points on the rectangle border, because that would imply $i=0$ or $j=0$. Figure~\ref{fig: S_gridpoints} illustrates a case where the maximum is not attained.
\end{proof}

\begin{proposition}
\label{prop: len ineq i j}
Consider $a, b, \len \in \mathbb R^+$ and $i, j \in \mathbb N$, $i, j \geq 2$ arbitrary.
\begin{enumerate}
\item
\label{prop: len ineq i j: ineqs}
The following inequalities hold for any segment with length $\len$ whose discrete bounding rectangle has normalized dimensions $i, j$:
\begin{align}
\label{eq: len > sqrt}
\len &> \sqrt{(i-2)^2 a^2 + (j-2)^2 b^2}, \\
\label{eq: len leq sqrt}
\len &\leq \sqrt{i^2 a^2 + j^2 b^2}.
\end{align}
\item
\label{prop: len ineq i j: exist}
Conversely, if $\len$, $i$, $j$ satisfy \eqref{eq: len > sqrt} and \eqref{eq: len leq sqrt} there exists a segment of length $\len$ whose discrete bounding rectangle has normalized dimensions $i$ and $j$.
\item
\label{prop: len ineq i j: ineq, exist}
There is a segment of length not exceeding $\len$ that has a discrete bounding rectangle with normalized dimensions $i, j$ if and only if \eqref{eq: len > sqrt} holds.
\end{enumerate}
\end{proposition}

\begin{proof}
\ref*{prop: len ineq i j: ineqs} The inequalities follow from the fact that the segment endpoints lie in the interiors or on the outer edges of two tiles in opposite corners of the discrete bounding rectangle. This is illustrated in Figure~\ref{fig: len ineq i j} for two specific $(i,j)$ pairs. For each $(i,j)$, segments are shown with lengths close to either of the two bounds. Note that inequality \eqref{eq: len > sqrt} is valid even for $i=2$, $j=2$, in which case it reduces to $\len>0$.

\ref*{prop: len ineq i j: exist} For $a$, $b$, $\len$, $i$, $j$ satisfying the two inequalities, a segment of length $\len$ can be found that has its endpoints in the interiors or on the outer edges of the two shaded tiles of a rectangle with normalized dimensions $i$ and $j$ (see Figure~\ref{fig: len ineq i j}), which is thus the discrete bounding rectangle of that segment.
\begin{figure}
\centering%
\subfigure[$i=4$, $j=3$]{%
\label{fig: len ineq 4 3}%
\includegraphics[width=.49\textwidth]{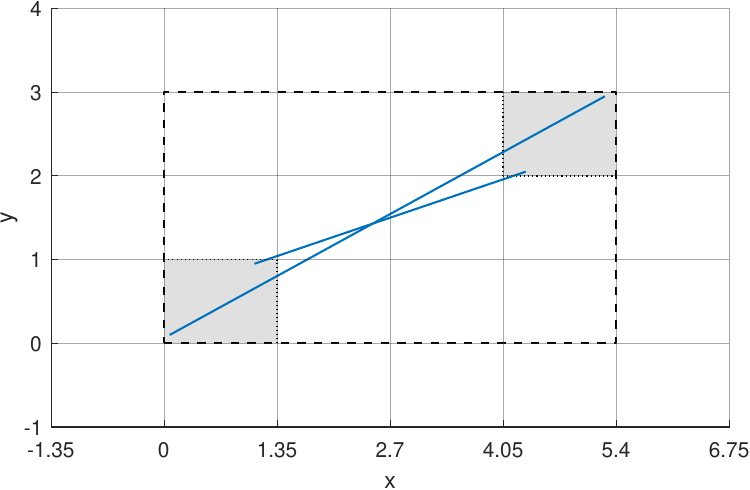}%
}\hfill%
\subfigure[$i=4$, $j=2$]{%
\label{fig: len ineq 4 2}%
\includegraphics[width=.49\textwidth]{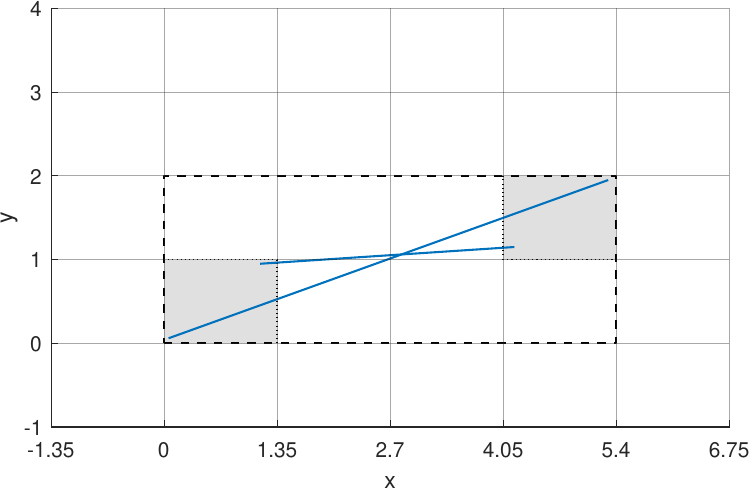}%
}%
\caption{Relationship between segment length $\len$ and dimensions $i$, $j$ of the discrete bounding rectangle
}%
\label{fig: len ineq i j}
\end{figure}%

\ref*{prop: len ineq i j: ineq, exist} ``\eqref{eq: len > sqrt} $\Rightarrow$ there is a segment\ldots'': Assume that \eqref{eq: len > sqrt} holds. It is always possible to choose a length equal to or smaller than $\len$ such that both \eqref{eq: len > sqrt} and \eqref{eq: len leq sqrt} hold. The result follows, for that length, from part~\ref{prop: len ineq i j: exist}.

``There is a segment\ldots $\Rightarrow$ \eqref{eq: len > sqrt}'': Assume that a segment exists with length $\len' \leq \len$ and with a discrete bounding rectangle of normalized dimensions $i$, $j$. From part~\ref{prop: len ineq i j: ineqs} it follows that inequality \eqref{eq: len > sqrt} holds for the length $\len'$, and thus for $\len$.
\end{proof}

Consider the problem of maximizing the number of visited tiles for a given length. According to Proposition~\ref{prop: i+j-1}, the position and orientation of the segment should be chosen to obtain $i+j-1$ as large as possible, where $i$ and $j$ are the normalized dimensions of its discrete bounding rectangle. On the other hand, Proposition~\ref{prop: len ineq i j} restricts the $i, j$ values that can be achieved with a given length. A relevant question is: are there any $(i,j)$ pairs that can be disregarded irrespective of the length $\len$? In other words, what is the ``smallest'' subset of $\mathbb N^2$ such that the $(i,j)$ pair that maximizes the number of tiles for any given length can always be found within that subset?

For instance, it is clear from Figure~\ref{fig: examples} that segment orientations near the vertical or horizontal directions (resulting in $i=1$ with large $j$, or $j=1$ with large $i$) will not maximize the number of visited tiles, and thus the corresponding $(i,j)$ pairs can be discarded. On the other hand, the set of optimal $(i,j)$ pairs must contain one such pair for each possible value of $i+j-1$, so that the set can produce that value as the solution (maximum number of visited tiles) for certain lengths. It is insightful to examine two specific examples (Figure~\ref{fig: ijLS}) before giving an explicit formula for the coordinates of the optimal pairs.

Consider $a=b=1$ first. This is illustrated in Figure~\ref{fig: ijLS_1}. Note that in this and in the next figures the axes represent $ia$ and $jb$ (not $i$ and $j$). In this graph, each dashed diagonal line contains points $(ia,jb)$ with the same $i+j-1$; and the lower bound \eqref{eq: len > sqrt} corresponds to an arc centered at $(2a,2b)$.

\begin{figure}
\centering%
\subfigure[$a=b$]{%
\label{fig: ijLS_1}
\includegraphics[width=.7\textwidth]{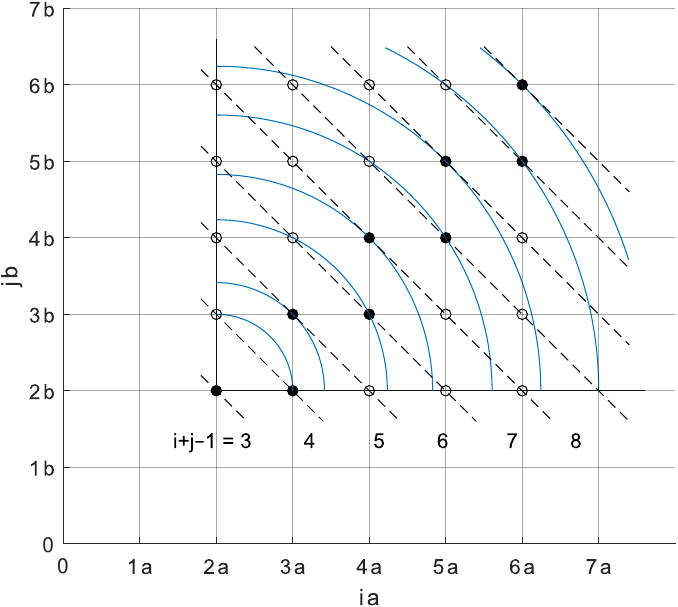}%
}\\%
\centering%
\subfigure[$a=1.35b$]{%
\label{fig: ijLS_1p35}
\includegraphics[width=.7\textwidth]{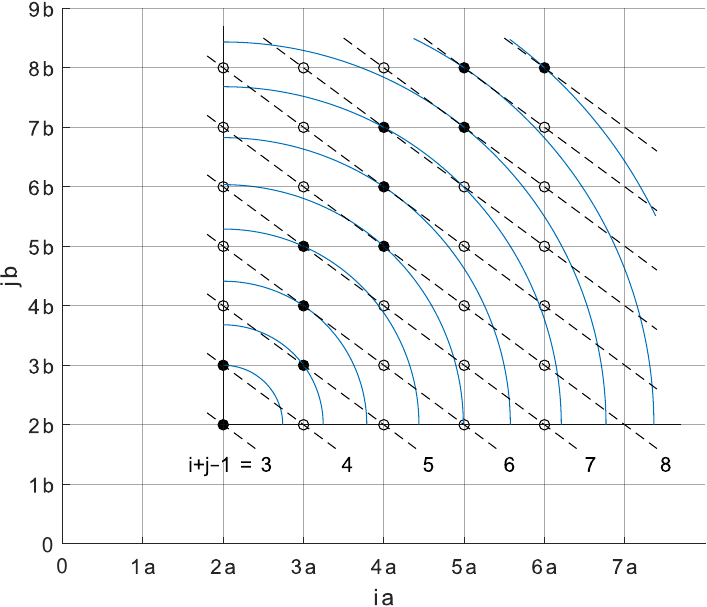}%
}%
\caption{Relationship of segment length and number of visited tiles with the width and height of the discrete bounding rectangle
}%
\label{fig: ijLS}%
\end{figure}%

For a given $\len \in \mathbb R^+$, the $(i,j)$ pairs that can be achieved with segments of length not exceeding $\len$ are, by Proposition~\ref{prop: len ineq i j}.\ref{prop: len ineq i j: ineq, exist}, those that satisfy \eqref{eq: len > sqrt}. Thus for a given value of $i+j-1$ the optimal $(i,j)$ pair (the one that can be achieved with the smallest length) is determined by the condition that the point $(ia,jb)$ minimizes the distance to $(2a,2b)$. Such pairs are depicted as filled circles in the figure, and the arcs represent the inequality \eqref{eq: len > sqrt} for each of the optimal pairs.

Conversely, given a length $\len$, the maximum number of visited tiles will be achieved with one of these pairs, namely the pair $(i,j)$ such that $(ia,jb)$ is on the uppermost (or rightmost) diagonal line while still being contained in the circle of radius $\len$ centered at $(2a,2b)$.

The optimal pairs in this specific case have the form $(i,i)$ or $(i,i-1)$, as seen in the figure. Due to symmetry, any pair $(i,i-1)$ could be replaced by $(i-1,i)$. This illustrates that the set of optimal pairs is not unique in general. 

As a second example, consider $a=1.35$, $b=1$. This is depicted in Figure~\ref{fig: ijLS_1p35}. Again, the optimal pair $(i,j)$ for each diagonal is that for which the point $(ia,jb)$ is closest to $(2a,2b)$; but in this case the $i$, $j$ coordinates of these pairs do not follow a rule as simple as in the previous example.

The following proposition gives an explicit method to obtain a set of optimal $(i,j)$ pairs. This set will be denoted as $\mss = \{(i_3,j_3), (i_4,j_4), \ldots\}$, where the pair $(i_\tiles,j_\tiles)$ corresponds to $i+j-1 = \tiles$.

\begin{proposition}
\label{prop: min suff set, form}
Given $a, b \in \mathbb R^+$, a set of optimal pairs $\mss = \{(i_3,j_3), (i_4,j_4), \ldots\}$ can be obtained as
\begin{align}
\label{eq: min suff set, form, i}
i_\tiles &= \left\lfloor \frac{(\tiles-3) b^2}{a^2+b^2} + \frac 5 2 \right\rfloor, \\
\label{eq: min suff set, form, j}
j_\tiles &= \left\lceil  \frac{(\tiles-3) a^2}{a^2+b^2} + \frac 3 2 \right\rceil,
\end{align}
where $i_\tiles + j_\tiles-1 = \tiles$. All pairs $(i_\tiles,j_\tiles)$ are strictly below the line
\begin{equation}
\label{eq: upper bound, line}
j = \frac{i a^2}{b^2} - \frac{3a^2}{2b^2} + \frac 5 2,
\end{equation}
and above or on the line
\begin{equation}
\label{eq: lower bound, line}
j = \frac{i a^2}{b^2} - \frac{5a^2}{2b^2} + \frac 3 2.
\end{equation}
\end{proposition}

\begin{proof}
For each $\tiles \geq 3$, the pair $(i_\tiles,j_\tiles)$ should be chosen as that on the line $i+j-1=\tiles$ which minimizes $(i-2)^2 a^2 + (j-2)^2 b^2$. This allows maximizing the sum $i+j-1$, and thus the number of visited tiles, for a given length restriction; or visiting a specified number of tiles with lengths as small as possible.

Consider, for the moment, $i$, $j$ as if they were real-valued variables, and denote $x=ia$, $y=jb$. The line $i+j-1 = \tiles$ then becomes
\begin{equation}
\label{eq: line 1}
\frac x a + \frac y b = \tiles + 1,
\end{equation}
and $(i-2)^2 a^2 + (j-2)^2 b^2$ is expressed as $(x-2a)^2+(y-2b)^2$. The point minimizing this quadratic function along the line \eqref{eq: line 1} is the intersection of the latter with the perpendicular line passing through $(2a,2b)$,
\begin{equation}
\label{eq: line 2}
y = \frac {a (x-2a)} b + 2b.
\end{equation}
An example with $a=1.35$, $b=1$, $\tiles=7$ is shown in Figure~\ref{fig: min suff set, form}, where \eqref{eq: line 2} is depicted as a solid line. Solving the system of equations \eqref{eq: line 1} and \eqref{eq: line 2} gives 
\begin{align}
\frac x a - 2 &= \frac{(\tiles-3)b^2}{a^2+b^2}, \\
\frac y b - 2 &= \frac{(\tiles-3)a^2}{a^2+b^2}.
\end{align}
In terms of the real-valued variables $i$, $j$, the solution $(\isolr, \jsolr)$ is thus
\begin{align}
\label{eq: prop: min suff set, form: isolr}
\isolr &= \frac{(\tiles-3)b^2}{a^2+b^2} + 2, \\
\label{eq: prop: min suff set, form: jsolr}
\jsolr &= \frac{(\tiles-3)a^2}{a^2+b^2} + 2.
\end{align}
The corresponding point $(\isolr a, \jsolr b)$ is shown in Figure~\ref{fig: min suff set, form} with a square marker.

\begin{figure}%
\centering%
\includegraphics[width=.7\textwidth]{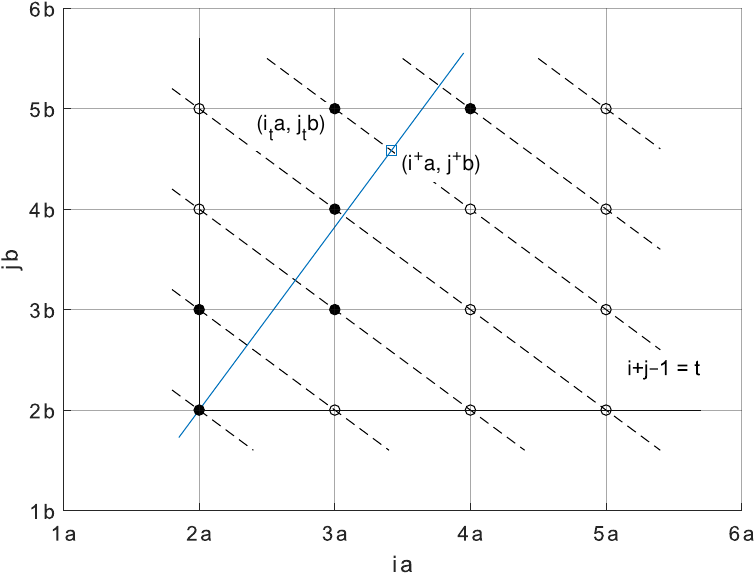}%
\caption{Obtaining $(\isolr, \jsolr)$ and $(i_\tiles, j_\tiles)$ in Proposition~\ref{prop: min suff set, form}. Example with $a=1.35$, $b=1$, $\tiles=7$%
}%
\label{fig: min suff set, form}%
\end{figure}%

The variables $i$, $j$ are actually limited to integer values. The pair $(i_\tiles, j_\tiles) \in \mathbb Z^2$ that minimizes $(i-2)^2 a^2 + (j-2)^2 b^2$ along the line $i+j-1=\tiles$ is either $(\lceil \isolr \rceil, \lfloor \jsolr \rfloor)$ or $(\lfloor \isolr \rfloor, \lceil \jsolr \rceil)$, whichever gives $(i_\tiles a, j_\tiles b)$ closest to $(\isolr a, \jsolr b)$, as illustrated in Figure~\ref{fig: min suff set, form}. In case of a tie the first of the two options is (arbitrarily) chosen. This can be expressed as
\begin{align}
\label{eq: prop: min suff set, form: isoli}
i_\tiles &= \left\lfloor \isolr + \frac 1 2 \right\rfloor, \\
\label{eq: prop: min suff set, form: jsoli}
j_\tiles &= \left\lceil \jsolr - \frac 1 2 \right\rceil,
\end{align}
which corresponds to rounding $\isolr$ and $\jsolr$ to the closest integers, with ties resolved in opposite directions. Combining \eqref{eq: prop: min suff set, form: isolr}--\eqref{eq: prop: min suff set, form: jsoli}
yields \eqref{eq: min suff set, form, i} and \eqref{eq: min suff set, form, j}.

From \eqref{eq: prop: min suff set, form: isolr} and \eqref{eq: prop: min suff set, form: isoli},
\begin{equation}
\label{eq: prop: min suff set, form: isoli ineq}
\frac{(\tiles-3)b^2}{a^2+b^2} + \frac 3 2 < i_\tiles \leq \frac{(\tiles-3)b^2}{a^2+b^2} + \frac 5 2,
\end{equation}
and similarly, from \eqref{eq: prop: min suff set, form: jsolr} and \eqref{eq: prop: min suff set, form: jsoli},
\begin{equation}
\label{eq: prop: min suff set, form: jsoli ineq}
\frac{(\tiles-3)a^2}{a^2+b^2} + \frac 3 2 \leq j_\tiles < \frac{(\tiles-3)a^2}{a^2+b^2} + \frac 5 2.
\end{equation}
Considering the first inequality in \eqref{eq: prop: min suff set, form: isoli ineq} and the second in \eqref{eq: prop: min suff set, form: jsoli ineq} as equalities and eliminating $\tiles$ gives \eqref{eq: upper bound, line}. The pair $(i_\tiles, j_\tiles)$ is strictly below the line \eqref{eq: upper bound, line} because the used inequalities are strict. Similarly, \eqref{eq: lower bound, line} results from the second inequality in \eqref{eq: prop: min suff set, form: isoli ineq} and the first in \eqref{eq: prop: min suff set, form: jsoli ineq}, and the fact that those inequalities are not strict implies that the bound \eqref{eq: lower bound, line} can actually be attained.
\end{proof}

The bounding lines in Proposition~\ref{prop: min suff set, form} are shown in Figure~\ref{fig: pairs_bounds}, using three different pairs of grid parameters $a$, $b$ as examples. Given $(i_\tiles,j_\tiles) \in \mss$, the next pair $(i_{\tiles+1},j_{\tiles+1})$ is obtained by incrementing $j$ if that results in a point below \eqref{eq: upper bound, line}. Else $i$ is incremented instead, and the new pair is guaranteed to be above or on \eqref{eq: lower bound, line}.

\begin{figure}
\centering%
\subfigure[$a = 1$, $b = 1$]{%
\label{fig: pairs_bounds_1}%
\includegraphics[height=.23\textheight]{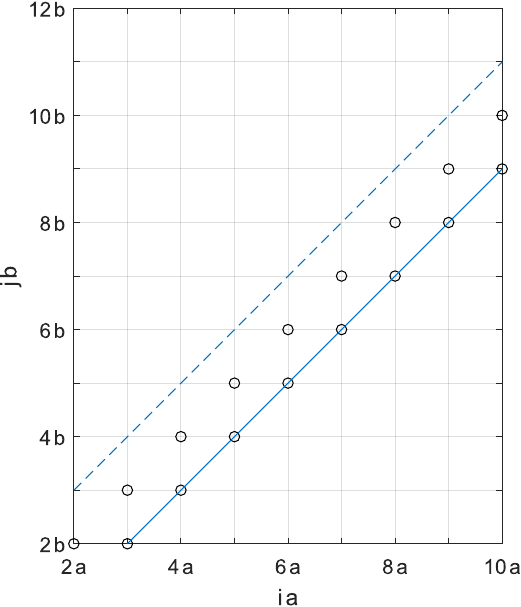}
}\hfill%
\subfigure[$a = 1.35$, $b = 1$]{%
\label{fig: pairs_bounds_1p35}%
\includegraphics[height=.23\textheight]{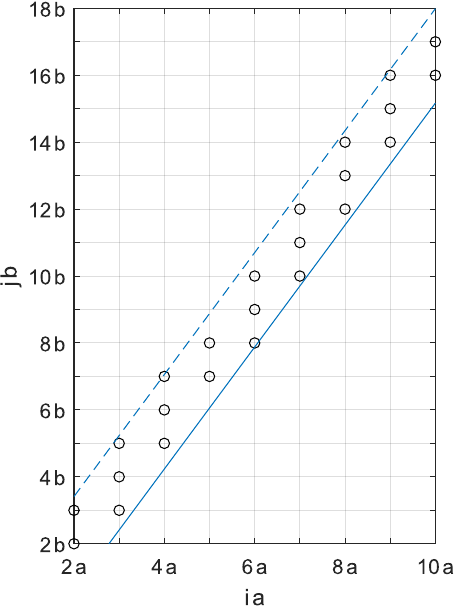}
}\hfill%
\subfigure[$a = \sqrt{2}$, $b = 1$]{%
\label{fig: pairs_bounds_sqrt2}%
\includegraphics[height=.23\textheight]{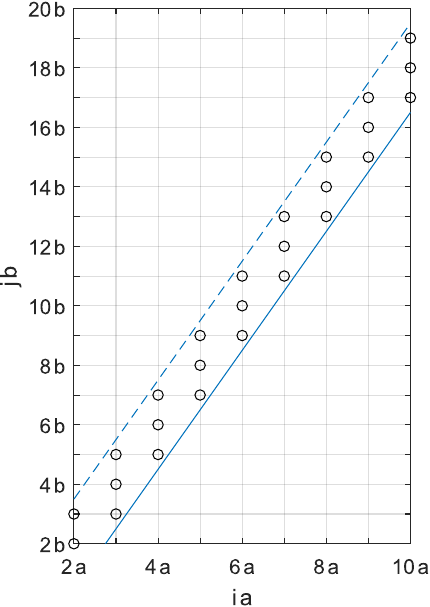}
}%
\caption{Set of optimal pairs $\mss$, and bounding lines
}%
\label{fig: pairs_bounds}%
\end{figure}%
 
For $a^2/b^2$ arbitrary, the number of pairs in $\mss$ with the same $i$, or with the same $j$, is in general irregular, because the lines \eqref{eq: upper bound, line} and \eqref{eq: lower bound, line} do not follow a ``natural'' direction of the grid. This happens for instance in Figure~\ref{fig: pairs_bounds_1p35}, where the number of pairs for each $i$ equals either $2$ or $3$ without a clear pattern.\footnote{Strictly, there is a periodic pattern whenever $a^2/b^2$ is rational, which is the case in Figure~\ref{fig: pairs_bounds_1p35}. However, the pattern is not easily discernible unless $a^2/b^2$ is a ratio of small numbers.}
On the other hand, a simple pattern arises when $a^2/b^2$ or $b^2/a^2$ is a natural number, as seen in Figures~\ref{fig: pairs_bounds_1} and \ref{fig: pairs_bounds_sqrt2}.

A segment whose discrete bounding rectangle has normalized width $i$ and height $j$ is oriented with approximate slope $jb/(ia)$ with respect to the $x$ axis (see Figure~\ref{fig: len ineq i j}); and this approximation becomes better for greater segment lengths. From \eqref{eq: upper bound, line} and \eqref{eq: lower bound, line} it can be seen that the pairs $(i,j) \in \mss$ have $j/i \approx a^2/b^2$ for large $i, j$. Therefore the optimal slope for long segments is approximately $a/b$. This substantiates the intuition that to maximize the number of visited tiles, the segment direction should strike a balance between achieving a small perceived ``length'' of the tile on one hand, and crossing both horizontal and vertical grid lines on the other hand.

\section{Deterministic segment: direct and inverse problems}
\label{part: max}

The direct and inverse problems defined in \S\ref{part: intro}, considering the segment position and orientation as deterministic, are addressed in this section. The general case for rectangular grids with real-valued segment lengths is analyzed first, in \S\ref{part: max: arbitrary grid, real lengths}. The square grid with real-valued segment lengths is addressed in \S\ref{part: max: unit square grid, real lengths}, as it allows a specialized formula for the direct problem. Lastly, the analysis of a unit square grid with integer-valued lengths is presented in \S\ref{part: max: unit square grid, integer lengths}.

\subsection{Arbitrary grid with real-valued lengths}
\label{part: max: arbitrary grid, real lengths}

Given a grid with parameters $a, b \in \mathbb R^+$, the maximum number $\tiles$ of visited tiles for an arbitrary real-valued length $\len$ can be represented by a function $\funt: \mathbb R^+ \to \mathbb N$ such that $\tiles = \funt(\len)$. Similarly, for the inverse problem a function $\funl: \mathbb N \to \mathbb R^+$ can be defined such that $\funl(\tiles)$ gives the infimum length of all segments that visit at least $\tiles$ tiles. Clearly, these two functions are related as
\begin{align}
\label{eq: funt funl}
\funt(\len) &= \max \{\tiles \in \mathbb N \st \funl(\tiles)<\len\},\\
\label{eq: funl funt}
\funl(\tiles) &= \inf\{\len \in \mathbb R^+ \st \funt(\len) \geq \tiles\}.
\end{align}

For arbitrary $a, b \in \mathbb R^+$, the functions $\funt$ and $\funl$ can be computed using an iterative procedure, which exploits the fact that the pairs $(i_3,j_3), (i_4,j_4), \ldots$ of the set $\mss$ are sorted by increasing $i+j-1$, and also by increasing $(i-2)^2 a^2 + (j-2)^2 b^2$. Namely, for $\funt$ the following method yields the solution: generate successive pairs to find the last one, $(i_\tiles,j_\tiles)$, that satisfies \eqref{eq: len > sqrt}; then $\funt(\len) = \tiles$. For $\funl$ the analogous method gives a direct formula. In addition, it is possible to obtain a direct formula also for $\funt$ using a different approach. These formulas are given in Theorems~\ref{theo: funt, form} and \ref{theo: funl, form}.

\begin{theorem}
\label{theo: funt, form}
For $a, b \in \mathbb R^+$, $a \geq b$ and $\len \in \mathbb R^+$,
\begin{equation}
\label{eq: theo: funt, form; funt}
\funt(\len) = \isoli+\jsoli-1
\end{equation}
with
\begin{align}
\label{eq: theo: funt, form; i}
\isoli &= \left\lceil \frac 3 2 + \frac b a \Re \sqrt{\frac{\len^2}{a^2+b^2}-\frac 1 4} \right\rceil, \\
\label{eq: theo: funt, form; j}
\jsoli &= \left\lceil 1 + \frac{\sqrt{\len^2-(\isoli-2)^2a^2}}{b} \right\rceil.
\end{align}
The function $\funt$ is piecewise constant and left-continuous, with unit-height jumps. A jump occurs at $\len$ if and only if $\len = \funl(\tiles)$ for some $\tiles \in \mathbb N$, $\tiles \geq 4$; and then $\funt(\len) = \tiles-1$, $\lim_{\delta \rightarrow 0+} \funt(\len+\delta) = \tiles$.
\end{theorem}

\begin{proof}
The approach is similar to that used in the proof of Proposition~\ref{prop: min suff set, form}. First, the intersection point $(\isolr a, \jsolr b)$, $\isolr, \jsolr \geq 2$ between the line defined by \eqref{eq: lower bound, line} and the arc centered at $(2a,2b)$ with radius $\len$ is computed, if it exists. Then, based on either the values $\isolr, \jsolr$ or the non-existence of the intersection point, a pair of integer values $(\isoli, \jsoli)$ is obtained that maximizes the $i+j-1$ sum that can be achieved with segments of length up to $\len$.

As will be seen, the obtained pair $(\isoli,\jsoli)$ may belong to the set $\mss$ defined by Proposition~\ref{prop: min suff set, form} or not. However, in either case $\funt(\len)$ is given by $\isoli+\jsoli-1$. The two possibilities are respectively illustrated in Figures~\ref{fig: ijLS_1p35_detail_in} and \ref{fig: ijLS_1p35_detail_notin} for $a=1.35$, $b=1$. In each case, the displayed arc is centered at $(2a,2b)$ and has radius $\len$. The inner region defined by the arc contains all $(ia,jb)$ points such that the pair $(i,j)$ is achievable according to Proposition~\ref{prop: len ineq i j}.\ref{prop: len ineq i j: ineq, exist}. As in previous figures, filled circles represent points $(ia,jb)$ such that $(i,j) \in \mss$. The solid line is the bound \eqref{eq: lower bound, line}. The intersection point $(\isolr a, \jsolr b)$ is displayed with a square marker.

\begin{figure}
\centering%
\subfigure[$\len = 3.1$: $(\isoli,\jsoli) \in \mss$]{%
\label{fig: ijLS_1p35_detail_in}%
\includegraphics[width=.7\textwidth]{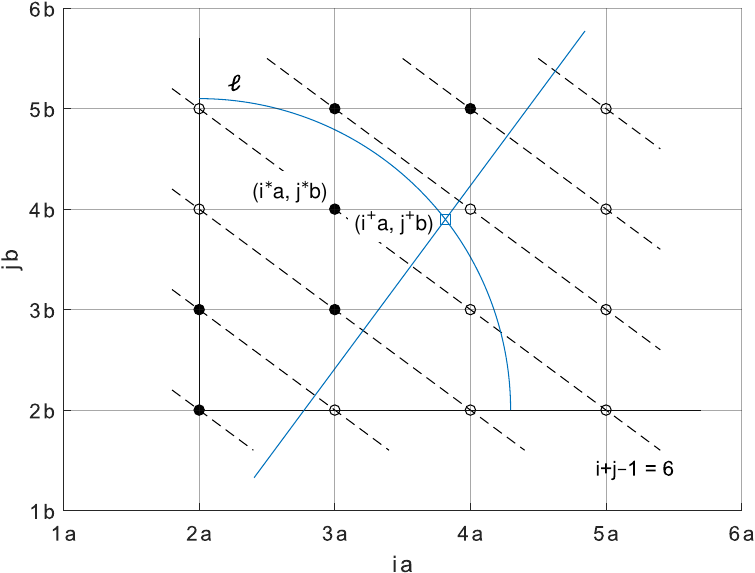}%
}\\
\subfigure[$\len = 3.7$: $(\isoli,\jsoli) \notin \mss$]{%
\label{fig: ijLS_1p35_detail_notin}%
\includegraphics[width=.7\textwidth]{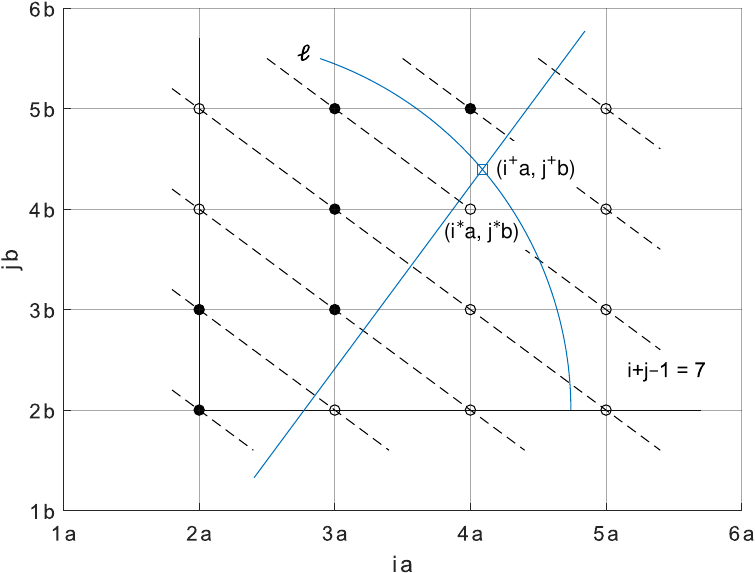}%
}%
\caption{Obtaining $(\isoli,\jsoli)$ in Theorem~\ref{theo: funt, form}. Examples with $a=1.35$, $b=1$%
}%
\label{fig: ijLS_1p35_detail}%
\end{figure}%

The pair $(\isolr, \jsolr)$ results from solving the equation system
\begin{align}
\label{eq: icont jcont 1}
(\isolr-2)^2 a^2 + (\jsolr-2)^2 b^2 &= \len^2, \\
\label{eq: icont jcont 2}
\jsolr = \frac{\isolr a^2}{b^2} - \frac{5a^2}{2b^2} + \frac 3 2.
\end{align}
Expressing these equations in terms of new variables $\isolr-5/2$ and $\jsolr-3/2$, the solutions are easily found to be
\begin{align}
\label{eq: isolr sol}
\isolr &= \frac 5 2 \pm \frac b a \sqrt{\frac{\len^2}{a^2+b^2} - \frac 1 4}, \\
\label{eq: jsolr sol}
\jsolr &= \frac 3 2 \pm \frac a b \sqrt{\frac{\len^2}{a^2+b^2} - \frac 1 4},
\end{align}
where the same sign (positive or negative) should be used in the two equations. This can yield zero, one or two real-valued solution pairs $(\isolr, \jsolr)$, which respectively corresponds to the solid line in Figure~\ref{fig: ijLS_1p35_detail} being exterior, tangent or secant to the circle (the figure depicts the latter situation).

A solution pair $(\isolr, \jsolr)$ given by \eqref{eq: isolr sol} and \eqref{eq: jsolr sol} is meaningful only if it is real-valued with $\isolr, \jsolr \geq 2$. This solution, if it exists, is always associated with the positive sign in those expressions. Since $a \geq b$, it is easily seen that $\jsolr \geq 2$ implies $\isolr \geq 2$, and thus it suffices to check the former condition. Three cases need to be distinguished: there are no real-valued solution pairs $(\isolr, \jsolr)$; there are one or two but none of them has $\jsolr \geq 2$; or there are one or two and and one of them satisfies that inequality. These cases correspond to different ranges of $\len$, as seen next.

For $\len < \sqrt{a^2+b^2}/2$, \eqref{eq: isolr sol} and \eqref{eq: jsolr sol} are not real-valued. Since $a \geq b$, from the inequality $\len < \sqrt{a^2+b^2}/2$ it follows that $\len < a/\sqrt{2} < a$. This means that any achievable $(i,j)$ pair, if any, will have $i=2$. Thus in this case $\isoli$ should be set to $2$.

For $\sqrt{a^2+b^2}/2 \leq \len < (a^2+b^2) / (2a)$, \eqref{eq: isolr sol} and \eqref{eq: jsolr sol} give either two real-valued solutions or one real-valued double solution for $(\isolr, \jsolr)$, with $\jsolr < 2$. This in turn implies, according to \eqref{eq: icont jcont 2}, that $\isolr < (b^2/a^2+5)/2 \leq 3$ for $a \geq b$. Thus only pairs with $i=2$ are achievable again for $\len$ in this range, and therefore $\isoli$ must be $2$.

Lastly, for $\len \geq (a^2+b^2) / (2a)$ the expressions \eqref{eq: isolr sol} and \eqref{eq: jsolr sol} with positive sign give $\isolr, \jsolr \geq 2$, and $\isoli$ should be taken as the greatest integer less than $\isolr$, i.e.~$\lceil \isolr \rceil-1$.
 
The three cases are unified, as can be easily checked, by taking the real part of the positive-sign version of \eqref{eq: isolr sol} and computing $\isoli = \lceil \isolr \rceil-1$. This yields \eqref{eq: theo: funt, form; i}. Once $\isoli$ is known, \eqref{eq: theo: funt, form; j} computes $\jsoli$ as the greatest integer such that $(\isoli a,\jsoli b)$ is within the circle with center $(2a,2b)$ and radius $\len$. This ensures that $(\isoli,\jsoli)$ is achievable with lengths less than $\len$.

To see that $\funt(\len) = \isoli+\jsoli-1$, the two situations stated at the outset need to be considered separately. The first possibility is that $(\isoli,\jsoli) \in \mss$ (upper part of Figure~\ref{fig: ijLS_1p35_detail}). Then, by construction  $(\isoli,\jsoli)$ maximizes $i+j-1$ among all achievable pairs of $\mss$, and is therefore optimal.

The second possibility is that $(\isoli,\jsoli) \notin \mss$ (Figure~\ref{fig: ijLS_1p35_detail_notin}). This happens when the pair from $\mss$ that has $i = \isoli$ ($(4,5)$ in the figure) is outside the circle, i.e.~it would require a length greater than $\len$. The selected $(\isoli,\jsoli)$ ($(4,4)$ in the figure), however, has the same $i+j-1$ sum as the pair from $\mss$ that ``should'' be used, which is $(\isoli-1,\jsoli+1)$ ($(3,5)$ in the figure); and therefore gives the same result. This is always the case, because  $(\isoli,\jsoli+1) \in \mss$ (it is above or on the bounding line) and $(\isoli,\jsoli) \notin \mss$ (it is below the line), and due to how $\mss$ has been constructed, this implies that $(\isoli-1,\jsoli+1) \in \mss$ and $(\isoli-1,\jsoli+k) \notin \mss$ for $k =2, 3, \ldots$. It follows that $(\isoli,\jsoli)$ is achievable and maximizes $i+j-1$, and thus $\isoli+\jsoli-1$ is the desired solution.

Therefore, regardless of whether $(\isoli,\jsoli)$ is in $\mss$ or not, \eqref{eq: theo: funt, form; i} and \eqref{eq: theo: funt, form; j} give $\isoli+\jsoli-1$ equal to $\funt(\len)$. This establishes \eqref{eq: theo: funt, form; funt}.

Interestingly, for the specific case that $a^2/b^2 = 2k-1$, $k \in \mathbb N$ the lower bounding line \eqref{eq: lower bound, line} becomes $j = (2k-1)i - 5k + 4$, which gives integer $j$ for integer $i$. This means that for each $i$ there is a pair $(i,j) \in \mss$ that is on that line (see for example Figure~\ref{fig: pairs_bounds_1}), and the case $(\isoli,\jsoli) \notin \mss$ never occurs.

As for the properties of $\funt$, it stems from \eqref{eq: theo: funt, form; funt}--\eqref{eq: theo: funt, form; j} that this function is piecewise constant and left-continuous. From the procedure described in the previous paragraphs for obtaining $(\isoli, \jsoli)$ it is clear that $\isoli+\jsoli-1$ increases in steps of $1$ when $\len$ is increased continuously; that is, $\funt$ has jumps of unit height.

Consider an arbitrary $\len$ such that for some $\tiles \in \mathbb N$, $\tiles \geq 4$
\begin{equation}
\label{eq: theo funt, form: prop 1}
\funl(\tiles)=\len.
\end{equation}
To see that $\funt$ has a jump at $\len$, assume for the sake of contradiction that $\funt$ is continuous at that point. Therefore $\funt$ is constant on an interval containing that point, which implies that $\funt(\len-\epsilon) = \tiles = \funt(\len+\epsilon)$ for some $\epsilon > 0$. This means that there exists a segment with length $\len-\epsilon$ that visits $\tiles$ tiles, and thus $\funl(\tiles) \leq \len-\epsilon < \len$, in contradiction with \eqref{eq: theo funt, form: prop 1}. Therefore $\funt$ is discontinuous (from the right) at $\len$. By definition of $\funl$, from \eqref{eq: theo funt, form: prop 1} it follows that
\begin{equation}
\label{eq: theo funt, form: prop 2}
\funt(\len) < \tiles
\end{equation}
and there exists $\epsilon > 0$ such that $\funt(\len+\delta) = \tiles$ for $0 < \delta < \epsilon$. This implies that
\begin{equation}
\label{eq: theo funt, form: prop 3}
\lim_{\delta \rightarrow 0+} \funt(\len+\delta) = \tiles,
\end{equation}
that is, $\funt$ has a jump at $\len$. In addition, since the jump has unit height, it stems from \eqref{eq: theo funt, form: prop 2} and \eqref{eq: theo funt, form: prop 3}  that $\funt(\len) = \tiles-1$.

Conversely, assume that $\funt$ has a jump from $\tiles-1$ to $\tiles$ at some $\len \in \mathbb R^+$. This means that \eqref{eq: theo funt, form: prop 2} and \eqref{eq: theo funt, form: prop 3} hold for those $\tiles$ and $\len$. From \eqref{eq: theo funt, form: prop 2} it follows that  $\funl(\tiles) \geq \len$. On the other hand, \eqref{eq: theo funt, form: prop 3} implies that $\funl(\tiles) \leq \len$. Thus $\funl(\tiles)=\len$.
\end{proof}

Although Theorem~\ref{theo: funt, form} requires $a \geq b$, the result could obviously be applied for $a < b$ by swapping the values of $a$ and $b$.

\begin{theorem}
\label{theo: funl, form}
For $a, b \in \mathbb R^+$ and $\tiles \in \mathbb N$,
\begin{equation}
\label{eq: theo: funl, form; funl}
\funl(\tiles) = \sqrt{(\isoli-2)^2 a^2 + (\jsoli-2)^2 b^2}
\end{equation}
with
\begin{align}
\label{eq: theo: funl, form; isolr}
\isolr &= \max\left\{ \frac{(\tiles-3) b^2}{a^2+b^2}, 0 \right\} + 2, \\
\label{eq: theo: funl, form; jsolr}
\jsolr &= \max\left\{ \frac{(\tiles-3) a^2}{a^2+b^2}, 0 \right\} + 2, \\
\label{eq: theo: funl, form; isoli}
\isoli &= \left\lfloor \isolr + \frac 1 2 \right\rfloor, \\
\label{eq: theo: funl, form; jsoli}
\jsoli &= \left\lceil \jsolr - \frac 1 2 \right\rceil.
\end{align}
Equivalently, for $\tiles \geq 3$,
\begin{equation}
\label{eq: theo: funl, form; funl 2}
\funl(\tiles) = \sqrt{\frac{(\tiles-3)^2 a^2 b^2}{a^2+b^2} + \roundterm^2(a^2+b^2)}
\end{equation}
with
\begin{equation}
\label{eq: roundterm}
\roundterm = |\isoli - \isolr| = |\jsoli - \jsolr|.
\end{equation}
This function is monotone increasing for $\tiles \geq 3$.
\end{theorem}

\begin{proof}
The $(i_\tiles,j_\tiles)$ pair in set $\mss$ defined in Proposition~\ref{prop: min suff set, form} corresponds to at most $\tiles$ visited tiles. By construction of this set, any segment that visits $\tiles$ tiles must have length greater than $\sqrt{(i_\tiles-2)^2a^2 + (j_\tiles-2)^2b^2}$. For $\tiles \geq 3$ the variables $\isoli$, $\jsoli$ computed in \eqref{eq: theo: funl, form; isolr}--\eqref{eq: theo: funl, form; jsoli}
coincide with $i_\tiles$, $j_\tiles$ as given by \eqref{eq: min suff set, form, i} and \eqref{eq: min suff set, form, j}, and therefore \eqref{eq: theo: funl, form; funl} gives the correct result. For $\tiles \in \{1, 2\}$ both \eqref{eq: theo: funl, form; isoli} and \eqref{eq: theo: funl, form; jsoli} equal $2$, and \eqref{eq: theo: funl, form; funl} gives $0$, which is again the correct result.

For $\tiles \geq 3$, the term $\sqrt{(i_\tiles-2)^2a^2 + (j_\tiles-2)^2b^2}$ can be interpreted geometrically as the distance between $(i_\tiles a, j_\tiles a)$ and $(2a, 2b)$. As can be seen with the help of Figure~\ref{fig: min suff set, form}, this distance is the hypotenuse of a right triangle whose other two sides extend from $(2a, 2b)$ to $(\isolr a, \jsolr b)$ and from $(\isolr a, \jsolr b)$ to $(i_\tiles a, j_\tiles a)$ respectively. Therefore,
\begin{equation}
\label{eq: distance as hypotenuse}
\begin{split}
(i_\tiles-2)^2a^2 + (j_\tiles-2)^2b^2 &= (\isolr-2)^2a^2 + (\jsolr-2)^2b^2 \\
&\quad + (i_\tiles-\isolr)^2a^2 + (j_\tiles-\jsolr)^2b^2.
\end{split}
\end{equation}
For $\tiles \geq 3$ it stems from \eqref{eq: theo: funl, form; isolr}
and \eqref{eq: theo: funl, form; jsolr} that
\begin{equation}
\label{eq: distance as hypotenuse, first leg}
(\isolr-2)^2a^2 + (\jsolr-2)^2b^2 =
\frac{(\tiles-3)^2 (a^2b^4+a^4b^2)}{(a^2+b^2)^2} =
\frac{(\tiles-3)^2 a^2b^2}{a^2+b^2}.
\end{equation}
The fact that both $(\isolr, \jsolr)$ and $(i_\tiles, j_\tiles)$ are on the line $i+j-1=\tiles$ implies that $\isolr+\jsolr = i_\tiles+ j_\tiles$. Taking into account that $i_\tiles=\isoli$ and $j_\tiles=\jsoli$, it stems that $\vert i_\tiles-\isolr \vert = \vert j_\tiles-\jsolr \vert = \roundterm$ with $\roundterm$ given by \eqref{eq: roundterm}. Consequently,
\begin{equation}
\label{eq: distance as hypotenuse, second leg}
(i_\tiles-\isolr)^2a^2 + (j_\tiles-\jsolr)^2b^2 = \roundterm^2 (a^2+b^2).
\end{equation}
Substituting \eqref{eq: distance as hypotenuse, first leg} and \eqref{eq: distance as hypotenuse, second leg} into \eqref{eq: distance as hypotenuse} and using \eqref{eq: theo: funl, form; funl} yields \eqref{eq: theo: funl, form; funl 2}.

The definition of $\funl$ implies that $\funl(\tiles) \geq \funl(\tiles-1)$ for any $\tiles \in \mathbb N$. On the other hand, by Theorem~\ref{theo: funt, form}, $\funt$ is piecewise constant and has a unit-height jump from $\tiles-1$ to $\tiles$ at $\funl(\tiles)$, $\tiles \in \mathbb N$, $\tiles \geq 4$. This implies that $\funl(\tiles) > \funl(\tiles-1)$ for $\tiles \geq 4$.
\end{proof}

The expression \eqref{eq: theo: funl, form; funl 2} allows a neat interpretation of $\funl(\tiles)$ (as stems from the arguments used in the proof of Theorem~\ref{theo: funl, form}). Namely, $\funl^2(\tiles)$ is the sum of the two terms that appear in that expression. The first term is the squared distance from $(2a, 2b)$ to the diagonal line defined by $i+j-1 = \tiles$; and the second term is additional squared distance incurred from rounding $i$, $j$ to integer values.

Theorems~\ref{theo: funt, form} and \ref{theo: funl, form} not only give the solutions $\funt(\len)$ and $\funl(\tiles)$ to the first two questions posed in \S\ref{part: intro}; they also provide a way to actually position a segment of length $\len$ or slightly greater than $\funl(\tiles)$, respectively, so that it visits $\funt(\len)$ or $\tiles$ tiles. Namely, for $\isoli$, $\jsoli$ computed as in the corresponding theorem, the segment should have its endpoints in the interiors of two tiles shifted $\isoli-1$ steps horizontally and $\jsoli-1$ steps vertically with respect to each other, with the exact position and orientation of the segment adjusted to avoid any grid points.

It is interesting to consider the following particular cases: $\len \gg a,b$; $a \gg b$; and $a=b$. Regarding the first, from \eqref{eq: theo: funt, form; funt}--\eqref{eq: theo: funt, form; j}
and from \eqref{eq: theo: funl, form; funl}--\eqref{eq: theo: funl, form; jsoli}
it is seen that for long segments the number of visited tiles and the segment length are approximately proportional, with
\begin{equation}
\label{eq: asympt slope}
\lim_{\len \rightarrow \infty} \frac{\funt(\len)}{\len}
= \lim_{\tiles \rightarrow \infty} \frac{\tiles}{\funl(\tiles)}
= \sqrt{1/a^2 + 1/b^2}.
\end{equation}

As for $a \gg b$, in this case the optimal discrete bounding rectangle has $\isoli = 2$, and $\jsoli$ as large as allowed by $\len$ (direct problem) or as required by $\tiles$ (inverse problem), corresponding to an almost vertical segment. In other words, for $a \gg b$ the length of the segment is best invested in increasing the number of tiles traversed vertically (but the segment should be slightly tilted to cross a vertical edge), and the asymptotic ratio \eqref{eq: asympt slope} is approximately $1/b$.

For $a=b$, either from symmetry considerations or particularizing the results in the above theorems it stems that the optimal orientation of the segment is close to $45^\circ$. This case will be dealt with in \S\ref{part: max: unit square grid, real lengths}, as it lends itself to simplified formulas.

Figure~\ref{fig: funt funl examples} shows the functions $\funt$ and $\funl$ for several pairs of grid parameters $a$, $b$. The graphs illustrate some of the observations of the previous paragraphs. Indeed, the asymptotic slope in Figure~\ref{fig: funt examples}, or the inverse of the asymptotic slope in Figure~\ref{fig: funl examples}, is approximately $\sqrt{2}$ for $a=b=1$; and it is roughly $1/b$ for the case $a=5,b=1$, or even for $a=5,b=1.5$ or $a=10,b=3$. Comparing the latter two cases it is also seen that scaling $a$, $b$ and $\len$ by the same factor does not alter $\funt(\len)$, and results in $\funl(\tiles)$ being scaled by that factor.

\begin{figure}
\centering%
\subfigure[Function $\funt$]{%
\label{fig: funt examples}%
\includegraphics[width=.49\textwidth]{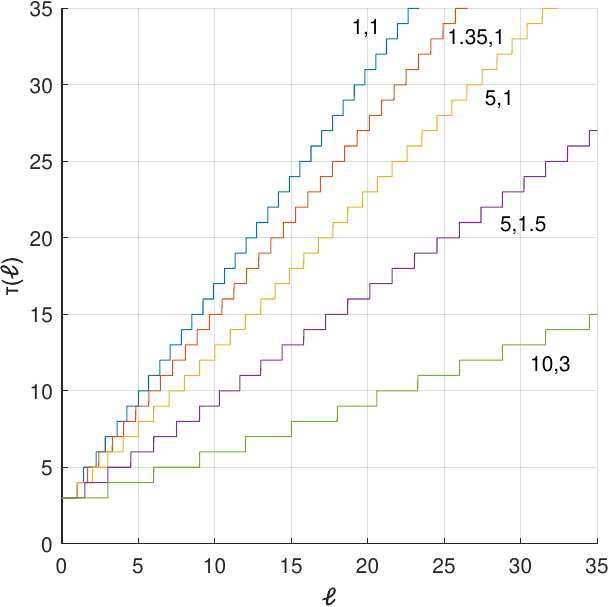}%
}\hfill%
\subfigure[Function $\funl$]{%
\label{fig: funl examples}%
\includegraphics[width=.49\textwidth]{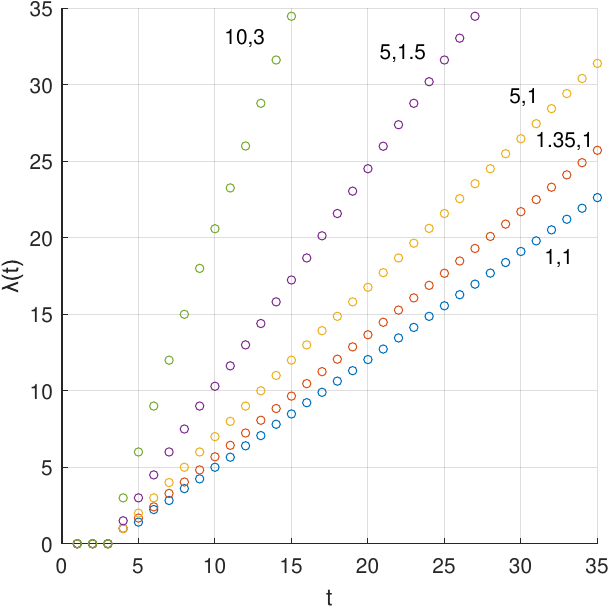}%
}%
\caption{Functions $\funt$ and $\funl$ for several pairs $a, b$
}%
\label{fig: funt funl examples}%
\end{figure}%

\subsection{Unit square grid with real-valued lengths}
\label{part: max: unit square grid, real lengths}

A square grid has $a=b$. For real-valued segment lengths it can be further assumed that $a=1$ (unit square grid). For $a \neq 1$ the expressions to be obtained are valid with $\len$ and $\funl(\tiles)$ replaced by $\len/a$ and $\funl(\tiles)/a$ respectively.

Particularizing the results in \S\ref{part: max: arbitrary grid, real lengths} to $a=b=1$ obviously yields simpler formulas.

\begin{corollary}
\label{cor: funt, sq, form}
For a unit square grid with $\len \in \mathbb R^+$,
\begin{equation}
\label{eq: cor: funt, sq, form}
\funt(\len) = \isoli+\jsoli-1
\end{equation}
with
\begin{align}
\label{eq: cor: funt, sq, form; i}
\isoli &= \left\lceil \frac 3 2 + \Re\sqrt{\frac{\len^2}{2} - \frac 1 4} \right\rceil, \\
\label{eq: cor: funt, sq, form; j}
\jsoli &= \left\lceil 1 + \sqrt{\len^2-(\isoli-2)^2} \right\rceil.
\end{align}
\end{corollary}

\begin{corollary}
\label{theo: funl, sq, form}
For a unit square grid, and for $\tiles \in \mathbb N$,
\begin{equation}
\funl(\tiles) = \begin{cases}
\displaystyle
0 & \text{for } \tiles =1, 2 \\[1.4mm] 
\displaystyle
\frac{\tiles-3}{\sqrt{2}} & \text{for } \tiles \text{ odd, } \tiles \geq 3 \\[4.5mm] 
\displaystyle
\frac{\sqrt{(\tiles-4)^2+(\tiles-2)^2}} {2} & \text{for } \tiles \text{ even, } \tiles \geq 4,
\end{cases}
\end{equation}
or equivalently
\begin{equation}
\label{eq: theo: funt, sq, form; funl}
\funl(\tiles) = \begin{cases}
\displaystyle
0 & \text{for } \tiles =1, 2 \\[1.4mm]  
\displaystyle
\sqrt{\left\lceil \frac{(\tiles-3)^2} {2} \right\rceil} & \text{for } \tiles \geq 3.
\end{cases}
\end{equation}
\end{corollary}

Furthermore, an even simpler formula can be obtained for $\funt$, as the next theorem shows.

\begin{theorem}
\label{theo: funt, sq, form}
For a unit square grid with $\len \in \mathbb R^+$,
\begin{equation}
\label{eq: theo: funt, sq, form; funt, with i, j}
\funt(\len) = \isoli+\jsoli-1
\end{equation}
with
\begin{align}
\label{eq: theo: funt, sq, form; i}
\isoli &= \left\lceil \frac{\len}{\sqrt{2}} \right\rceil + 1, \\
\label{eq: theo: funt, sq, form; j}
\jsoli &= \left\lceil \sqrt{\len^2-(\isoli-2)^2} \right\rceil + 1,
\end{align}
which gives either $\jsoli = \isoli$ or $\jsoli = \isoli+1$. Equivalently,
\begin{equation}
\label{eq: theo: funt, sq, form; funt, simpler}
\funt(\len) = \left\lfloor \sqrt{2 \left\lceil \len^2 \right\rceil - 2} \right\rfloor + 3.
\end{equation}
\end{theorem}

\begin{proof}
The proof of \eqref{eq: theo: funt, sq, form; funt, with i, j}--\eqref{eq: theo: funt, sq, form; j}
uses a variation of the set $\mss$ defined in Proposition~\ref{prop: min suff set, form} that is more suited to this situation.

For $a=b=1$, the set $\mss$ consists of points of the form $(i,i)$ and $(i,i-1)$, as is easily seen from Proposition~\ref{prop: min suff set, form}, and as illustrated in Figure~\ref{fig: pairs_bounds_1}. By symmetry, replacing each point $(i,i-1)$ by $(i-1,i)$ gives a set $\mss'$ that is also optimal. For this new set, the lower bounding line \eqref{eq: lower bound, line} can be replaced by the simpler $j=i$. The same approach followed in the proof of Theorem~\ref{theo: funt, form} can be applied here, but using this line. Thus $(\isolr, \jsolr)$ is obtained from
\begin{align}
(\isolr-2)^2 + (\jsolr-2)^2  &= \len^2, \\
\jsolr &= \isolr,
\end{align}
which gives
\begin{equation}
\label{eq: theo: funt, sq, form; proof 1}
\isolr = \jsolr = \len/\sqrt{2} + 2.
\end{equation}
As in Theorem~\ref{theo: funt, form}, $\isoli$ is obtained as $\lceil \isolr \rceil-1$; and then $\jsoli$ is chosen as the largest integer such that $(\isoli, \jsoli)$ is achievable, i.e.~its distance from $(2, 2)$ is less than $\len$. The resulting $\isoli$ and $\jsoli$ are given by \eqref{eq: theo: funt, sq, form; i} and \eqref{eq: theo: funt, sq, form; j}.

The above procedure for choosing $\jsoli$ given $\isoli = \lceil \isolr \rceil-1$ always results in $\jsoli$ being either $\isoli$ or $\isoli+1$. This can be seen as follows. If $\jsoli = \isoli$, the point $(\isoli, \jsoli)$ is closer to $(2,2)$ than $(\isolr, \jsolr)$ is, and is therefore achievable. This implies that values of $\jsoli$ smaller than $\isoli$ will never be chosen. On the other hand, $\jsoli = \isoli+2$ or larger values are not achievable, because they would produce a sum $i+j-1$ greater than $\isolr+\jsolr$, which is impossible.

The preceding analysis shows that the pair $(\isoli,\jsoli)$ is in $\mss'$ and maximizes $i+j-1$. Therefore \eqref{eq: theo: funt, sq, form; funt, with i, j} holds.

To show \eqref{eq: theo: funt, sq, form; funt, simpler}, it is first noted that for $\tiles \geq 3$ Corollary~\ref{theo: funl, sq, form} gives
\begin{equation}
\funl(\tiles) = \sqrt{\left\lceil \frac{(\tiles-3)^2} {2} \right\rceil}.
\end{equation}
According to \eqref{eq: funt funl}, $\funt(\len)$ is obtained as the largest positive integer $\tiles$ such that 
\begin{equation}
\label{eq: theo: funt, sq, form; proof 2}
\left\lceil \frac{(\tiles-3)^2} {2} \right\rceil < \len ^ 2.
\end{equation}
Since the left-hand side of \eqref{eq: theo: funt, sq, form; proof 2} is an integer, the condition of being strictly less than $\len^2$ is equivalent to
\begin{equation}
\left\lceil \frac{(\tiles-3)^2} {2} \right\rceil  \leq \left\lceil \len^2 \right\rceil  - 1,
\end{equation}
which in turn is the same as
\begin{equation}
\frac{(\tiles-3)^2} {2}\leq  \left\lceil \len^2 \right\rceil  - 1. 
\end{equation}
Solving for $\tiles$ gives
\begin{equation}
\label{eq: theo: funt, sq, form; proof 3}
\tiles \leq \sqrt{2 \left\lceil \len^2 \right\rceil -2 } + 3.
\end{equation}
The desired quantity $\funt(\len)$, that is the largest positive integer $\tiles$ satisfying \eqref{eq: theo: funt, sq, form; proof 3}, is thus the right-hand side rounded down, as given by \eqref{eq: theo: funt, sq, form; funt, simpler}.
\end{proof}

From Theorem~\ref{theo: funt, sq, form} it stems that odd values of $\funt(\len)$ correspond to $\isoli = \jsoli$, whereas even values are achieved with $\jsoli = \isoli+1$. In addition, noting that $\funt(\len) = \tiles$ is equivalent to $\funl(\tiles) < \len \leq \funl(\tiles+1)$ and using Corollary~\ref{theo: funl, sq, form} the following characterization of $\funt$ is obtained. For $\tiles \geq 3$ with $\tiles$ odd, $\funt(\len) = \tiles$ if and only if
\begin{equation}
\label{eq: square, odd max tiles, lengths}
\len \in \left( \frac{\tiles-3}{\sqrt{2}}, \frac{ \sqrt{(\tiles-3)^2+(\tiles-1)^2}}{2} \right].
\end{equation}
For $\tiles \geq 4$, $\tiles$ even, $\funt(\len) = \tiles$ if and only if
\begin{equation}
\label{eq: square, even max tiles, lengths}
\len \in \left(\frac{ \sqrt{(\tiles-4)^2+(\tiles-2)^2}}{2}, \frac{\tiles-2}{\sqrt{2}} \right].
\end{equation}

\subsection{Unit square grid with integer lengths}
\label{part: max: unit square grid, integer lengths}

A natural variation of the direct and inverse problems introduced in \S\ref{part: intro} is to consider $a=b=1$ with the additional restriction that the segment length can only be a positive integer (equivalently, the square grid has spacing $a$ and the segment lengths are restricted to integer multiples of $a$).

The direct problem in this setting corresponds to the restriction of $\funt$ to $\mathbb N$. This will be denoted as a function $\funti: \mathbb N \to \mathbb N$ for greater clarity, although obviously $\funti(\leni) = \funt(\leni)$ for all $\leni \in \mathbb N$. The sequence $\funti(\leni)$, $\leni \in \mathbb N$ takes values $3, 5, 7, 8, 9, \ldots$, and is depicted in Figure~\ref{fig: funti}. This is A346232 in the On-Line Encyclopedia of Integer Sequences \cite{OEIS_unitsq_int_dir}. For this sequence, the expression \eqref{eq: theo: funt, sq, form; funt, simpler} in Theorem~\ref{theo: funt, sq, form} simplifies in the obvious way, and the following properties hold.

\begin{figure}%
\centering%
\subfigure[Sequence $\funti$]{%
\label{fig: funti}%
\includegraphics[width=.49\textwidth]{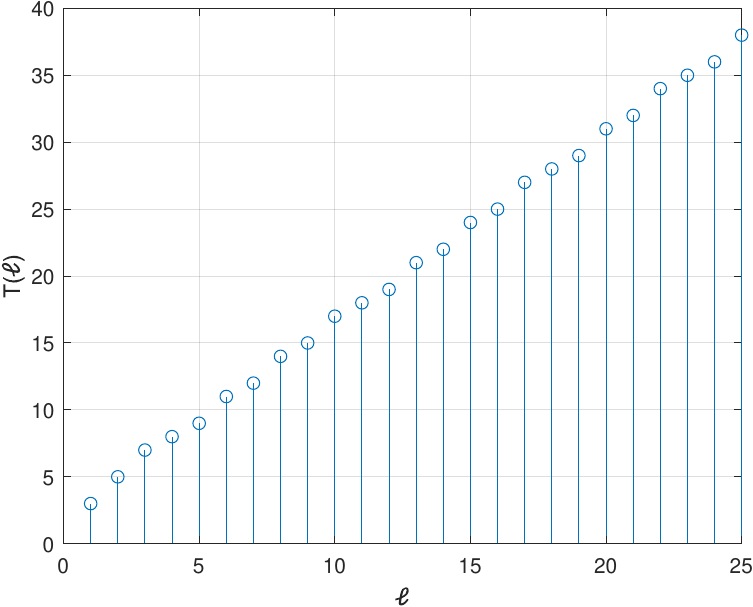}%
}\hfill%
\subfigure[Sequence $\funli$]{%
\label{fig: funli}%
\includegraphics[width=.49\textwidth]{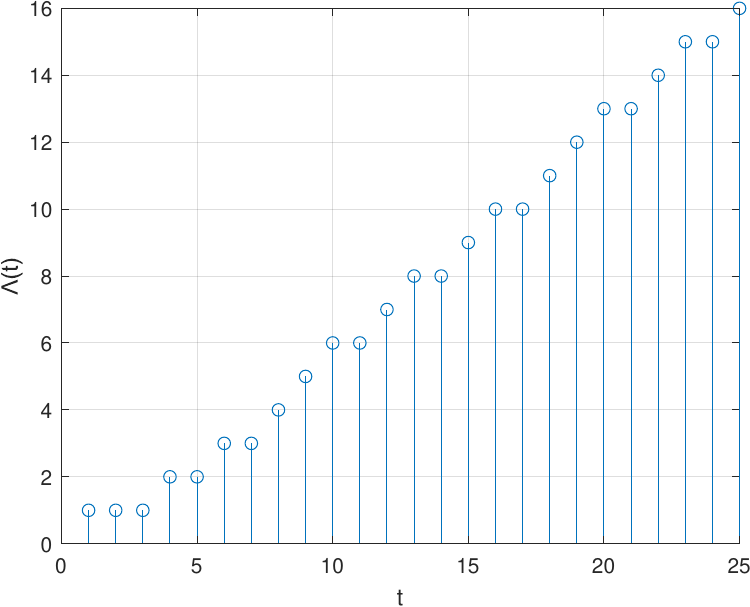}%
}%
\caption{Sequences $\funti$ and $\funli$
}%
\label{fig: funti funli}%
\end{figure}%

\begin{theorem}
\label{theo: funti}
For $\leni \in \mathbb N$,
\begin{equation}
\label{eq: funti leni}
\funti(\leni) = \left\lfloor \sqrt{2\leni^2-2} \right\rfloor + 3.
\end{equation}
In addition,
\begin{enumerate}
\item
This sequence is increasing, with $\funti(\leni+1)-\funti(\leni) \in \{1, 2\}$.
\item
There can be no more than $2$ consecutive increments equal to $1$.
\item
Increments equal to $2$ always appear isolated, except at the initial sequence terms $3, 5, 7$.
\end{enumerate}
\end{theorem}

\begin{proof}
The equality \eqref{eq: funti leni} stems from \eqref{eq: theo: funt, sq, form; funt, simpler} noting that $\len$ is an integer.

In order to prove that $\funti(\leni+1)-\funti(\leni) \in \{1,2\}$, consider the function $\genfun(\genvar) = \sqrt{2\genvar^2-2}$ for $\genvar \in \mathbb R$, $\genvar>1$. Its first derivative is
\begin{equation}
\genfun'(\genvar) = \frac {\sqrt{2} \, \genvar} {\sqrt{\genvar^2-1}},
\end{equation}
and its second derivative is easily seen to be negative. Therefore $\genfun'(\genvar)$ can be bounded for $\genvar > 3$ as
\begin{equation}
\label{eq: der bounds}
\lim_{\genvar \rightarrow \infty} \genfun'(\genvar) = \sqrt{2} < \genfun'(\genvar) < \genfun'(3) = 3/2.
\end{equation}
For $\leni \geq 3$,
by the mean value theorem \cite[section~5.3]{Abbott15}, when $\leni$ is increased to $\leni+1$ the term $\sqrt{2\leni^2-2}$ in \eqref{eq: funti leni} has an increment that equals $\genfun'(\genvar)$ for some $\leni < \genvar < \leni+1$. Therefore
\begin{equation}
\label{eq: der bounds 2}
\sqrt{2} < \sqrt{2(\leni+1)^2-2} - \sqrt{2\leni^2-2} < 3/2.
\end{equation}
Since $1 < \sqrt{2}$ and $3/2 < 2$, \eqref{eq: der bounds 2} implies that $\funti(\leni+1)-\funti(\leni)$ can only take the values $1$ or $2$ for $\leni \geq 3$. In addition, $\funti(2)-\funti(1) = \funti(3)-\funti(2) = 2$, and thus the result holds for all $\leni \in \mathbb N$.

Using the first bound in \eqref{eq: der bounds 2} three times,
\begin{equation}
3\sqrt{2} < \sqrt{2(\leni+3)^2-2} - \sqrt{2\leni^2-2}.
\end{equation}
Considering that $4 < 3\sqrt{2}$, this implies that $\funti(\leni+3)-\funti(\leni) \geq 4$ for $\leni \geq 3$. Therefore at least one of the three increments from $\funti(\leni)$ to $\funti(\leni+3)$ is $2$. Since $\funti(2)-\funti(1) = \funti(3)-\funti(2) = 2$, this result holds for all $\leni \in \mathbb N$.

Similarly, using the second bound in \eqref{eq: der bounds 2} twice,
\begin{equation}
\sqrt{2(\leni+2)^2-2} - \sqrt{2\leni^2-2} < 3,
\end{equation}
which implies that $\funti(\leni+2)-\funti(\leni) \leq 3$ for $\leni \geq 3$. Therefore the two increments $\funti(\leni+1)-\funti(\leni)$ and $\funti(\leni+2)-\funti(\leni+1)$ cannot both be $2$ for $\leni \geq 3$.
\end{proof}

The inverse problem with integer-length segments can be formulated as follows: given $\tiles \in \mathbb N$, find the \emph{minimum} integer length that allows visiting at least $\tiles$ tiles. Observe that in this case, unlike with real-valued lengths, there is indeed a minimum length, as every subset of $\mathbb N$ has a minimum. This can be expressed as a function $\funli: \mathbb N \to \mathbb N$:
\begin{equation}
\label{eq: funli funti}
\funli(\tiles) = \min\{\leni \in \mathbb N \st \funti(\leni) \geq \tiles\},
\end{equation}
which is related to the function $\funl$ for real-valued lengths by
\begin{equation}
\label{eq: funli funl}
\funli(\tiles) = \lfloor\funl(\tiles)\rfloor + 1.
\end{equation}
The converse to \eqref{eq: funli funti} is (compare to \eqref{eq: funt funl}):
\begin{equation}
\label{eq: funti funli}
\funti(\leni) = \max \{\tiles \in \mathbb N \st \funli(\tiles) \leq \leni\}.
\end{equation}
In view of \eqref{eq: funli funti} and \eqref{eq: funti funli}, $\funti$ and $\funli$ can be considered as ``pseudo-inverse'' sequences of each other.

The sequence $\funli(\tiles)$, $\tiles \in \mathbb N$ can be computed using \eqref{eq: theo: funt, sq, form; funl} and \eqref{eq: funli funl}. It has initial values $1, 1, 1, 2, 2, 3, 3, 4, 5 \ldots$, as seen in Figure~\ref{fig: funli}. This is A346693 in the On-Line Encyclopedia of Integer Sequences \cite{OEIS_unitsq_int_inv}. Moreover, a slightly simpler expression can be obtained from \eqref{eq: funti leni} and \eqref{eq: funli funti}. This is established by the next theorem, which also states some properties of $\funli$, parallel to those of $\funti$.

\begin{theorem}
\label{theo: funli}
For $\tiles \in \mathbb N$,
\begin{equation}
\label{eq: funli}
\funli(\tiles) = \begin{cases}
\displaystyle
1 & \text{for } \tiles \leq 3 \\[1.3mm] 
\displaystyle
\left \lceil \sqrt{\frac{(\tiles-3)^2} 2 + 1} \ \right \rceil & \text{for } \tiles \geq 4.
\end{cases}
\end{equation}
In addition,
\begin{enumerate}
\item
This sequence is non-decreasing. Except for the initial run of $3$ equal values, it is formed by runs of $1$ or $2$ equal values, with an increment of $1$ between consecutive runs.
\item
There can be no more than $3$ consecutive terms that are different.
\item
A run of $2$ equal values always has $2$ different terms before and $2$ different terms after the run, except for the initial terms $1, 1, 1, 2, 2, 3, 3$.
\end{enumerate}
\end{theorem}

\begin{proof}
Using \eqref{eq: funti leni}, the inequality $\funti(\leni) \geq \tiles$ in \eqref{eq: funli funti} is written as
\begin{equation}
\left\lfloor \sqrt{2\leni^2-2} \right\rfloor + 3 \geq \tiles.
\end{equation}
Since the right-hand side is an integer, this is equivalent to
\begin{equation}
\label{eq: theo funli ineq}
\sqrt{2\leni^2-2} \geq \tiles-3.
\end{equation}
Assuming $\tiles \geq 4$, taking squares and rearranging gives
\begin{equation}
\leni \geq \sqrt{\frac{(\tiles-3)^2} 2 + 1},
\end{equation}
which combined with \eqref{eq: funli funti} yields the second part of \eqref{eq: funli}. The first part results from noting that for $\tiles \leq 3$ the value $\leni=1$ satisfies \eqref{eq: theo funli ineq}.

The stated properties for $\funli$ follow directly from those of $\funti$ established by Theorem~\ref{theo: funti}.
\end{proof}

\section{Random segment: probabilistic characterization}
\label{part: rand}

Given $\len \in \mathbb R^+$, consider a segment of length $\len$ with uniformly random position and orientation. Specifically, the coordinates $x_1, y_1$ of the first endpoint are independent random variables uniformly distributed on $[0,a)$ and $[0,b)$ respectively, where $a$, $b$ are the grid parameters. The orientation $\theta$ of the segment is uniformly distributed on $[0,2\pi)$. The variables $x_1$, $y_1$ and $\theta$ determine the coordinates $x_2, y_2$ of the second endpoint.

Each realization of the random segment gives rise to a discrete bounding rectangle, whose normalized dimensions $i$ and $j$ are thus random variables, as is the number $\tiles$ of visited tiles. Except for a set of realizations with probability $0$, $i$ and $j$ are at least $1$, and $\tiles = i+j-1$. Note that $i$ and $j$ are not statistically independent.

This section deals with the two problems stated in \S\ref{part: intro} for random segments, namely obtaining the average number of visited tiles and the probability of visiting the maximum possible number of tiles. Segment lengths will be assumed to be real-valued. The results to be obtained are directly applicable for integer lengths as a particular case.

Arbitrary grids are considered in \S\ref{part: probmax: arbitrary grid, real lengths}. The main result is the average number of visited tiles. The square grid is addressed in \S\ref{part: probmax: unit square grid, real lengths}. This more specialized setting allows computation of the probability of visiting the maximum number of tiles, which would be difficult in the general case.

\subsection{Arbitrary grid with real-valued lengths}
\label{part: probmax: arbitrary grid, real lengths}

Let $\funta: \mathbb R^+ \to \mathbb R^+$ be defined such that $\funta(\len)$ is the average number of tiles visited by a random segment of length $\len$, with the distributions specified in the preceding.

\begin{theorem}
\label{theo: funta, form}
Given $a, b, \len \in \mathbb R^+$, consider a grid with parameters $a, b$ and a uniformly random segment of length $\len$, as defined above. The average number of tiles visited by the segment is 
\begin{equation}
\label{eq: funta, form}
\funta(\len) = \frac{2\len}{\pi}\left(\frac 1 a + \frac 1 b\right) + 1.
\end{equation}
\end{theorem}

\begin{proof}
Suppose first that the grid is simplified to only vertical lines with spacing $a$. This matches the set-up of Buffon's original needle problem \cite[section~1.1]{Mathai99}, except that here the length $\len$ of the needle may exceed the spacing $a$, allowing it to cross multiple lines. As shown in \cite{Ramaley69}, the expected number of lines crossed equals $2\len/(\pi a)$.

Consider again a grid with horizontal spacing $a$ and vertical spacing $b$. The grid crossings decompose into crossings of horizontal and vertical grid lines. By linearity of expectation, the expected number of crossings is the sum of the expectations for parallel lines with spacing $a$ and $b$ respectively, which gives
\[
\frac{2\len}{\pi}\left(\frac 1 a + \frac 1 b\right).
\]
As noted in Proposition~\ref{prop: i+j-1}, the number of tiles visited by a segment is the count of its grid line crossings plus $1$, unless it exactly passes through any grid points, but that occurs with probability $0$ and therefore does not affect the expected value. Thus $\funta(\len)$ is obtained by adding $1$ to the above expression, which gives \eqref{eq: funta, form}.
\end{proof}

In view of Theorem~\ref{theo: funta, form}, the average number of visited tiles as a function of the segment length has a very simple form, namely an affine function. Conversely, for any $\tiles>1$ it is immediate to compute the length of a random segment that visits $\tiles$ tiles on average, given as $\funta^{-1}(\tiles)$.

In spite of the dependence between the random variables $i$ and $j$, their marginal distributions have relatively simple analytic expressions, as established by the next proposition.

For $0 \leq \genvar \leq 1$, let
\begin{equation}
\label{eq: f}
f(\genvar) = \int_0^\genvar \arccos x \, \diff x = \genvar \arccos \genvar - \sqrt{1-\genvar^2} + 1.
\end{equation}

\begin{proposition}
\label{prop: Pr i}
Given $a, b, \len \in \mathbb R^+$, consider a grid with parameters $a, b$ and a uniformly random segment of length $\len$, as defined above. Let the random variables $i$, $j$ represent the normalized dimensions of the discrete bounding rectangle. For $n \in \mathbb N$,
\begin{equation}
\label{eq: Pr i geq n}
\Pr[i \geq n] = \begin{cases}
\displaystyle
1 &\text{if\ \ } \displaystyle n =1 \\[1 mm] 
\displaystyle
\frac{2\len}{\pi a} \left(f\left(\frac{a(n-1)}{\len}\right)-f\left(\frac{a(n-2)}{\len} \right)\right) &\text{if\ \ } \displaystyle 2 \leq n < \frac \len a + 1 \\[4 mm] 
\displaystyle
\frac{2\len}{\pi a}\,\, \left(1 - f\left(\frac{a(n-2)}{\len}\right)\right) &\text{if\ \ } \displaystyle \frac\len a + 1 \leq n < \frac\len a+2 \\[3 mm] 
\displaystyle
0 &\text{if\ \ } \displaystyle \frac \len a + 2 \leq  n;
\end{cases}
\end{equation}
and $\Pr[j \geq n]$ is given by the same expressions with $a$ replaced by $b$.
\end{proposition}

\begin{proof}
Clearly, $\Pr[i \geq 1]=1$. In the following it will be assumed that $n \geq 2$. The basic idea is to compute $\Pr[i \geq n]$ conditioned on $(x_1,y_1)$ (or, as will be seen, only on $x_1$), and then to average over $x_1$ and $y_1$ (actually only over $x_1$).

Given the coordinates $(x_1, y_1)$ of the first endpoint of the segment, with $0 \leq x_1 < a$, $0 \leq y_1 < b$, the second endpoint $(x_2, y_2)$ lies on a circle with radius $\len$ centered at $(x_1,y_1)$, as shown in Figure~\ref{fig: Pr_i}. The segment orientation is a random angle $\theta$ uniformly distributed on $[0,2\pi)$. It is clear from the figure that $i \geq n$ if and only if $x_2 > a(n-1)$ or $x_2 < -a(n-2)$; and for $n \geq 2$ these events are exclusive. Thus
\begin{equation}
\label{eq: Pr i cond}
\Pr[i \geq n \cond x_1, y_1] = \Pr[x_2 > a(n-1) \cond x_1, y_1] + \Pr[x_2 < -a(n-2) \cond x_1, y_1].
\end{equation}
The two conditional probabilities on the right-hand side of \eqref{eq: Pr i cond} are different in general. However, averaging over $x_1, y_1$ gives, by symmetry, $\Pr[x_2 > a(n-1)] = \Pr[x_2 < -a(n-2)]$. In addition, the coordinate $y_1$ does not have any influence on these events, and therefore conditioning on $x_1, y_1$ is the same as conditioning on $x_1$. This implies that, for $n \geq 2$,
\begin{equation}
\label{eq: Pr i = 2 Pr}
\Pr[i \geq n] = 2 \Pr[x_2 > a(n-1)].
\end{equation}

\begin{figure}
\centering%
\subfigure[$x_2$ can exceed $a(n-1)$ for all $x_1$, $0 \leq x_1 \leq a$]{%
\label{fig: Pr_i_full}%
\includegraphics[width=.58\textwidth]{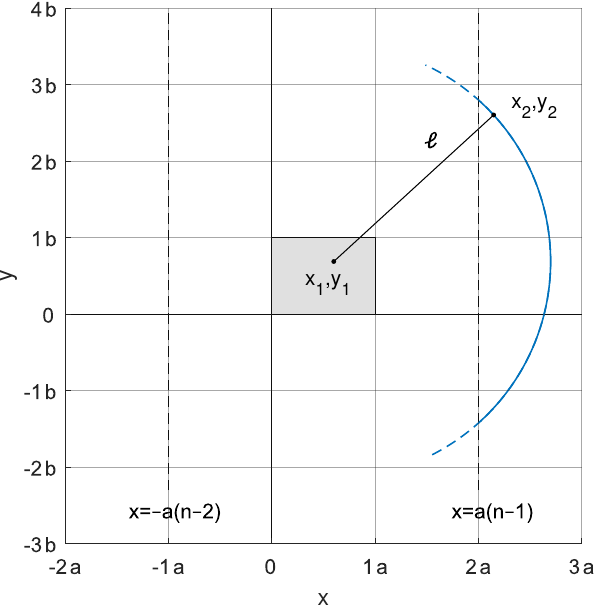}%
}\\
\subfigure[$x_2$ can exceed $a(n-1)$ only if $a(n-1)-\len \leq x_1 \leq a$]{%
\label{fig: Pr_i_notfull}%
\includegraphics[width=.58\textwidth]{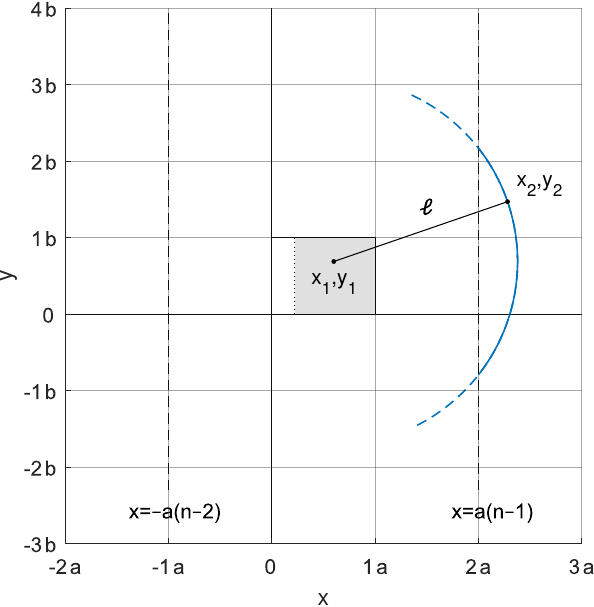}%
}%
\caption{Conditions for the normalized width of the discrete bounding rectangle, $i$, to be equal or greater than a given $n$. Example with $a=1.35$, $b=1$, $n=3$
}%
\label{fig: Pr_i}
\end{figure}%

Consider the event $x_2 > a(n-1)$ conditioned on $x_1$, with $n \geq 2$. There are three possibilities depending on $x_1$, $n$ and $\len$. If $a(n-1) < \len$, regardless of $x_1$ the length $\len$ is enough for $x_2$ to exceed $a(n-1)$ for some angles $\theta$. This is depicted in Figure~\ref{fig: Pr_i_full}, where the section of the arc with solid line represents, for a given $x_1$, those angles for which $x_2 > a(n-1)$. If $a(n-2) < \len \leq  a(n-1)$, the length will be enough provided that $x_1 > a(n-1)-\len$, and then only for certain angles. This restriction on $x_1$ corresponds to the shaded region in Figure~\ref{fig: Pr_i_notfull}. Lastly, if $\len \leq a(n-2)$ it is not possible for $x_2$ to exceed $a(n-1)$, regardless of $x_1$ or $\theta$. The figure makes it clear that the coordinate $y_1$ is irrelevant to this.

In the first two cases above, the probability that $x_2 > a(n-1)$, conditioned on $x_1$, is the length of the arc to the right of the line $x=a(n-1)$ divided by $2\pi\len$, that is,
\begin{equation}
\label{eq: Pr x_2 > a(n-1) cond}
\Pr[x_2 > a(n-1) \cond x_1] = \frac 1 \pi \arccos \frac{a(n-1)-x_1}{\len}.
\end{equation}
In the first case $x_1$ has a uniform distribution on $[0,a)$, and $\Pr[x_2 > a(n-1)]$ is easily obtained from \eqref{eq: Pr x_2 > a(n-1) cond} as
\begin{equation}
\label{eq: Pr x_2 > a(n-1), first case}
\begin{split}
\Pr[x_2 > a(n-1)] &= \E[\Pr[x_2 > a(n-1) \cond x_1]] \\
&= \frac 1 {\pi a} \int_0^a\arccos \frac{a(n-1)-x_1}{\len} \, \diff x_1 \\
&= \frac{\len}{\pi a}\,\, \left( f\left( \frac{a(n-1)}{\len} \right) - f\left( \frac{a(n-2)}{\len} \right) \right),
\end{split}
\end{equation}
where the function $f$ is defined in \eqref{eq: f}. Substituting into \eqref{eq: Pr i = 2 Pr} yields the result in \eqref{eq: Pr i geq n}, second line.

The second case is similar, but the integration over $x_1$ is from $a(n-1)-\len$ to $a$. Noting that $f(1)=1$, this gives
\begin{equation}
\label{eq: Pr x_2 > a(n-1), second case}
\Pr[x_2 > a(n-1)] = \frac{\len}{\pi a}\,\, \left( 1 - f\left( \frac{a(n-2)}{\len} \right) \right),
\end{equation}
which combined with \eqref{eq: Pr i = 2 Pr} yields the expression in \eqref{eq: Pr i geq n}, third line.

The third case obviously gives $\Pr[i \geq n] = 0$, as in \eqref{eq: Pr i geq n}, fourth line.

The above arguments can be applied to $\Pr[j \geq n]$ if the $x$ and $y$ axes are interchanged. Thus the formulas are the same with $a$ replaced by $b$.
\end{proof}

The results in Theorem~\ref{theo: funta, form} and Proposition~\ref{prop: Pr i} make clear the relationship between the problem considered here and Buffon's needle experiment, which deals with a random segment of fixed length in a regular structure of parallel strips \cite[section~1.1]{Mathai99}. Firstly, since the number $\pi$ is involved in \eqref{eq: funta, form}, it is possible to design a simple probabilistic experiment to estimate the value of $\pi$, as in Buffon's original problem. For example, choosing $a=b=\len=1$ gives $\funta(\len) = 4/\pi+1$.

Secondly, a grid with $b \rightarrow \infty$ corresponds to Buffon's arrangement of parallel strips with spacing $a$. Thus
\begin{equation}
\label{eq: number of lines crossed, Buffon}
\lim_{b \rightarrow \infty} \funta(\len)-1 = \frac{2\len}{\pi a}
\end{equation}
gives the average number of lines crossed in Buffon's experiment. For $\len \leq a$ the segment can cross at most one line, and \eqref{eq: number of lines crossed, Buffon} coincides with the probability of crossing \cite[theorem 1.1.1]{Mathai99}.

Lastly, $\Pr[i \geq 2]$ as computed in Proposition~\ref{prop: Pr i} can be interpreted as the probability of crossing at least one line in Buffon's experiment with arbitrary needle length. Indeed, for $n=2$ the third expression in \eqref{eq: Pr i geq n} reduces again to $2\len/(\pi a)$ for $\len \leq a$, whereas the second reduces to $2\len f(a/\len)/(\pi a)$ for $\len > a$. Using \eqref{eq: f} this is expressed as
\[
\frac 2 \pi \arccos\frac a \len + \frac{2\len}{\pi a} \left(1 - \sqrt{1-\frac{a^2}{\len^2}} \right),
\]
in agreement with the known result for Buffon's needle problem with $\len > a$ \cite[theorem 1.1.2]{Mathai99}. On the other hand, $\Pr[i \geq n]$ for $n \geq 3$ can be seen as the probability of crossing at least one line in a modified version of Buffon's set-up in which a needle endpoint is only allowed to move in a region of width $a$ located in the middle of a strip of width $(2n-3)a$. 

For $\len \rightarrow \infty$, the average number of visited tiles $\funta(\len)$ has the following asymptotic slope, as stems from \eqref{eq: funta, form}:
\begin{equation}
\label{eq: asympt slope ave}
\lim_{\len \rightarrow \infty} \frac{\funta(\len)}{\len} = \frac 2 \pi \left(\frac 1 a + \frac 1 b \right).
\end{equation}
On the other hand, the maximum number of visited tiles $\funt(\len)$ has an asymptotic slope given by \eqref{eq: asympt slope}. It is interesting to consider the ratio of these values, which also gives the asymptotic ratio between the average and the maximum numbers of visited tiles. This depends only on $a/b$, and is thus expressed as a function $\ras: \mathbb R^+ \to (0,1)$:
\begin{equation}
\label{eq: ras}
\ras(a/b) = \lim_{\len \rightarrow \infty}
\frac { \funta(\len) } { \funt(\len) }
= \frac
{ \lim_{\len \rightarrow \infty} (\funta(\len)/\len) }
{ \lim_{\len \rightarrow \infty} (\funt(\len)/\len) }
= \frac{2\left(1 + a/b\right)}{\pi\sqrt{1 + \left(a/b\right)^2}}.
\end{equation}
This function is represented in Figure~\ref{fig: ratio asympt slopes}. It satisfies $\ras(x) = \ras(1/x)$, and is therefore symmetric in logarithmic scale, as seen in the figure. The function $\ras$ is unimodal with maximum value $2\sqrt 2/\pi \approx 0.9003$ for $a/b=1$, and its limit when $a/b$ tends to $0$ or $\infty$ is $2/\pi \approx 0.6366$. This implies that for $\len$ large the average number of tiles $\funta(\len)$ cannot be very small compared with the maximum number of tiles $\funt(\len)$. For example, $\funta(\len)$ exceeds $0.8\funt(\len)$ asymptotically when $\len \rightarrow \infty$ for $1/3 \leq a/b \leq 3$.

\begin{figure}
\centering%
\includegraphics[width=.7\textwidth]{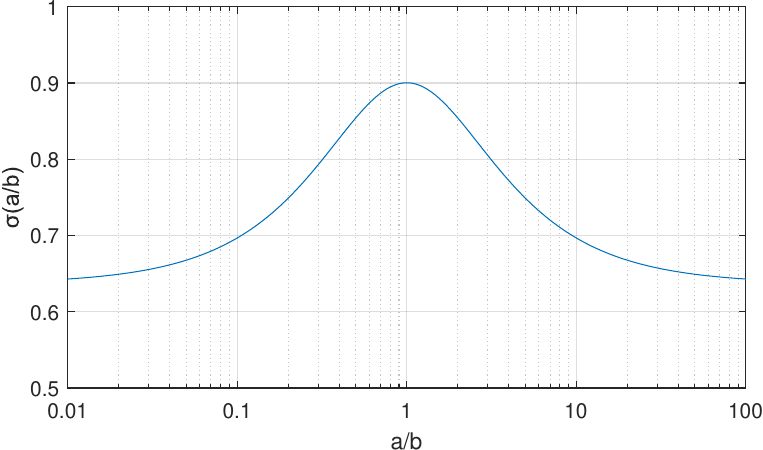}%
\caption{Asymptotic ratio of average and maximum numbers of visited tiles
}%
\label{fig: ratio asympt slopes}%
\end{figure}%

The probability of visiting the maximum number of tiles for a rectangular grid is difficult to compute. The reason is the irregularity of the relevant $(i,j)$ pairs for general $a$, $b$, analogous to that observed in \S\ref{part: fund results} for the pairs $(i,j)$ that maximize the number of visited tiles. In a square grid, however, the problem is more tractable.

\subsection{Unit square grid with real-valued lengths}
\label{part: probmax: unit square grid, real lengths}

Consider a square grid with unit spacing, $a=b=1$. Again, the results to follow can be applied to a square grid with spacing $a \neq 1$ if $\len$ is replaced by $\len/a$.

For a random segment with length $\len \in \mathbb R^+$ on a unit square grid, let the function $\probmax: \mathbb R^+ \to [0,1]$ be defined such that $\probmax(\len)$ gives the probability that the segment visits the maximum possible number of tiles, $\funt(\len)$. This function is characterized by the next theorem.

For $u, v \in \mathbb N$, $\genvar \in \mathbb R^+$, let
\begin{equation}
\label{eq: g def}
g(\genvar, u, v) = \frac 2 \pi \int_{\arcsin \frac v {\genvar}}^{\arccos \frac u {\genvar}} \left( \genvar \cos \theta - u \right) \left( \genvar \sin \theta - v \right) \, \diff \theta.
\end{equation}
Computing this integral and using the identities $\cos^2 z = 1-\sin^2 z$ and $\cos(2z) = 2\cos^2z-1 = 1 - 2\sin^2z$,
\begin{equation}
\label{eq: g fin}
\begin{split}
g(\genvar, u, v) &= \frac 1 {\pi} \biggl(
2\left(\arccos\frac{u}{\genvar}-\arcsin\frac{v}{\genvar}\right) u v + \genvar^2 + u^2 + v^2 \\
& \quad - 2 u \sqrt{\genvar^2-v^2} - 2 v \sqrt{\genvar^2-u^2} \biggr).
\end{split}
\end{equation}
Taking into account that $\arcsin x = \pi/2 - \arccos x$, it follows from \eqref{eq: g fin} that $g(\genvar, u, v) = g(\genvar, v, u)$.

\begin{theorem}
\label{theo: probmax, sq}
Given $\len \in \mathbb R^+$, consider a unit square grid and a uniformly random segment of length $\len$, as previously defined. The probability $\probmax(\len)$ that the segment visits the maximum number of tiles $\tiles= \funt(\len)$ is
\begin{equation}
\label{eq: probmax impl}
\probmax(\len) = \sum_{\substack{i,j \geq 2,\ %
i+j-1=\tiles,\\
(i-2)^2+(j-2)^2<\len^2
}} g\left(\len, i-2, j-2 \right),
\end{equation}
or equivalently, defining $w = \sqrt{\len^2-(\tiles-3)^2/2}$,
\begin{equation}
\label{eq: probmax expl}
\probmax(\len) =
\begin{cases}
\displaystyle
g\left(\len, \frac{\tiles-3} 2, \frac{\tiles-3} 2 \right) +
2 \sum_{k=1}^{\left\lceil \frac w {\sqrt{2}} - 1 \right\rceil} g\left(\len, \frac{\tiles-3} 2 - k, \frac{\tiles-3} 2 + k \right) &\text{for $\tiles$ odd} \\[6mm] 
\displaystyle
2 g\left(\len, \frac{\tiles-4} 2, \frac{\tiles-2} 2 \right) +
2 \sum_{k=1}^{\left\lceil \frac w {\sqrt{2}} - \frac 3 2 \right\rceil} g\left(\len, \frac{\tiles-4} 2 - k, \frac{\tiles-2} 2 + k \right) &\text{for $\tiles$ even}.
\end{cases}
\end{equation}
\end{theorem}

\begin{proof}
The approach is analogous to that used in the proof of Proposition~\ref{prop: Pr i}, but conditioning on the segment orientation $\theta$ instead of on the location of its first endpoint.


Without loss of generality, the first endpoint of the segment can be assumed to be contained in the tile with lower-left corner $(0,0)$, which will be called the reference tile.

Consider $\len \in \mathbb R^+$, and let $\tiles = \funt(\len)$. By Proposition~\ref{prop: i+j-1}, the segment visits $\tiles$ tiles almost surely if the dimensions $i$, $j$ of its discrete bounding rectangle are such that $i+j-1=\tiles$. In addition, a segment that visits the maximum number of tiles always has $i, j \geq 2$. Namely, if $i=1$ the segment can always be shifted and possibly tilted to cross a vertical grid line, thus increasing $i$ to $2$, without reducing $j$. The reasoning for $j=1$ is analogous. Thus for $i, j \geq 2$, $i+j-1=\tiles$ let $\probmax_{i,j}(\len)$ denote the probability that the segment has a discrete bounding rectangle with dimensions $i, j$. Then
\begin{equation}
\label{eq: probmax len probmax i,j len}
\probmax(\len) =  \sum_{i,j \geq 2,\ i+j-1=\tiles} \probmax_{i,j}(\len).
\end{equation}

For any $i, j \geq 2$ with $i+j-1=\tiles$, the segment has a discrete bounding rectangle with dimensions $i, j$ if and only if its second endpoint is in the tile with lower-left corner $(i-1,j-1)$ or in one of the other three symmetrical tiles with respect to the lines $x=1/2$ and $y=1/2$, i.e.~those with lower-left corners $(i-1,1-j)$, $(1-i,j-1)$ and $(1-i,1-j)$. An example is shown in Figure~\ref{fig: probmax_symmetric_tiles}. Consequently, $\probmax_{i,j}(\len)$ can be obtained by computing the probability that the second endpoint is in the tile with lower-left corner $(i-1,j-1)$ and multiplying by $4$.

\begin{figure}
\centering%
\includegraphics[width=.53\textwidth]{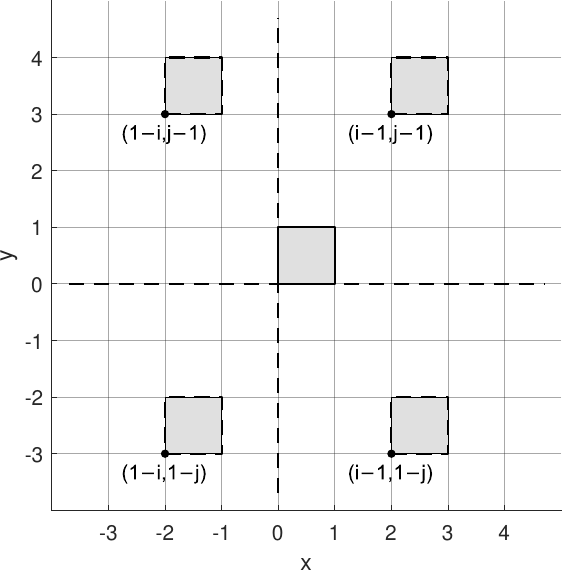}%
\caption{Symmetrical tiles for the computation of $\probmax_{i,j}(\len)$ on a unit square grid. Example for $i=3$, $j=4$%
}%
\label{fig: probmax_symmetric_tiles}%
\end{figure}%

By Proposition~\ref{prop: len ineq i j}, the tile with lower-left corner $(i-1,j-1)$ can be visited only if
\begin{equation}
\label{eq: i-2 j-2 len}
(i-2)^2+(j-2)^2<\len^2.
\end{equation}
Thus $\probmax_{i,j}(\len)=0$ if this restriction does not hold. In the following it will be assumed that it does.

For a fixed orientation $\theta$, the segment goes from the reference tile to that with lower-left corner $(i-1,j-1)$ if and only if shifting the segment up and to the right so that its first endpoint is at $(1,1)$ results in the second endpoint still being in the tile with lower-left corner $(i-1,j-1)$. This is illustrated in Figure~\ref{fig: probmax_funt_odd}. It cannot be the case that the segment ``overshoots'' past this tile, because that would yield a number of visited tiles exceeding $\tiles$, which is the maximum. In other words, only one corner of the tile can be contained within the arc in the figure, namely its lower-left corner $(i-1,j-1)$.

\begin{figure}
\centering
\includegraphics[width=.56\textwidth]{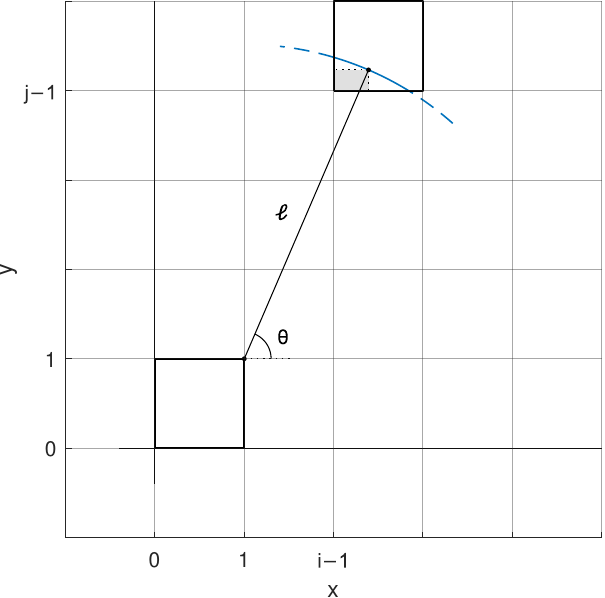}%
\caption{Segment orientations and positions for a given $\tiles = \funt(\len)$ on a unit square grid. Example for $\tiles = 7$
}%
\label{fig: probmax_funt_odd}%
\end{figure}%

There is a range of values of $\theta$ for which the above condition is satisfied. This corresponds to the part of the arc with solid line in Figure~\ref{fig: probmax_funt_odd}. For each $\theta$ in this range, the valid positions for the second endpoint of the segment are in a rectangle contained in the tile with lower-left corner $(i-1,j-1)$. These are the positions that cause the first endpoint to be in the reference tile. The rectangle associated with an example $\theta$ is shown in Figure~\ref{fig: probmax_funt_odd}, shaded. The area of this rectangle is the probability $\probmax_{i, j}(\len,\theta)$ that the segment goes from the reference tile to that with lower-left corner $(i-1,j-1)$ conditioned on $\theta$. As seen in the figure, this area is
\begin{equation}
\label{eq: probmax: probmax theta odd}
\probmax_{i, j}(\len, \theta) = \left(\len\cos\theta - i+2\right) \left(\len\sin\theta - j + 2\right).
\end{equation}
The range of allowed values for $\theta$, as can also be deduced from the figure, is $(\theta_0, \theta_1)$ with
\begin{align}
\label{eq: odd theta 0}
\theta_0 & = \arcsin\frac{j-2}{\len}, \\
\label{eq: odd theta 1}
\theta_1 & = \arccos\frac{i-2}{\len}.
\end{align}
For $\theta$ outside of this range, the conditional probability $\probmax_{i, j}(\len, \theta)$ is $0$. Averaging $\probmax_{i, j}(\len, \theta)$ over $\theta$ gives the probability that the segment visits this tile, and then multiplying by $4$ yields $\probmax_{i,j}(\len)$:
\begin{equation}
\begin{split}
\label{eq: probmax: cond on theta}
\probmax_{i,j}(\len) &= 4 \int_{\theta_0}^{\theta_1} \frac{\probmax_{i, j}(\len, \theta)}{2\pi} \, \diff \theta \\
&= \frac 2 {\pi} \int_{\arcsin\frac{j-2}{\len}}^{\arccos\frac{i-2}{\len}} \left( \len\cos\theta - i + 2 \right)  \left( \len\sin\theta - j + 2 \right) \, \diff \theta
= g( \len, i-2, j-2).
\end{split}
\end{equation}

The above arguments show that $\probmax_{i,j}(\len)$ is given by \eqref{eq: probmax: cond on theta} if the condition \eqref{eq: i-2 j-2 len} holds, and equals $0$ otherwise. Using this into \eqref{eq: probmax len probmax i,j len} yields \eqref{eq: probmax impl}.

Given $\tiles$, the range of lengths $\len$ such that $\tiles = \funt(\len)$ is either \eqref{eq: square, odd max tiles, lengths} or \eqref{eq: square, even max tiles, lengths}. For $\len$ in this range, the number of reachable tiles with lower-left corners $(i-1,j-1)$, $i,j \geq 2$, $i+j-1=\tiles$, i.e.~the number of terms in \eqref{eq: probmax impl}, depends on $\len$. If $\tiles$ is odd, the minimum number is $1$, corresponding to a tile with its lower-left corner on the line $x=y$, which gives the first term in the first expression in \eqref{eq: probmax expl}. For $\tiles$ even, the minimum number is $2$, corresponding to two symmetrical tiles whose lower-left $x$ and $y$ coordinates differ by $1$, which give the first term in the second expression. Either for $\tiles$ odd or even, additional tiles may become reachable as $\len$ grows (within the range allowed by $\tiles$), and these always occur in symmetrical pairs. These tiles (if any) correspond to the terms in the (possibly empty) sum indexed by $k$ in either expression in \eqref{eq: probmax expl}. 

Figure~\ref{fig: probmax_terms} contains examples for several values of $\tiles$, with a few values of $\len$ for each $\tiles$. Each length is represented by an arc with radius $\len$ centered at $(1,1)$. For a given $\tiles$, the filled circular markers are the lower-left corners of the tiles that are always reachable, whereas the empty circles correspond to tiles whose reachability depends on $\len$.

\begin{figure}
\centering
\includegraphics[width=.7\textwidth]{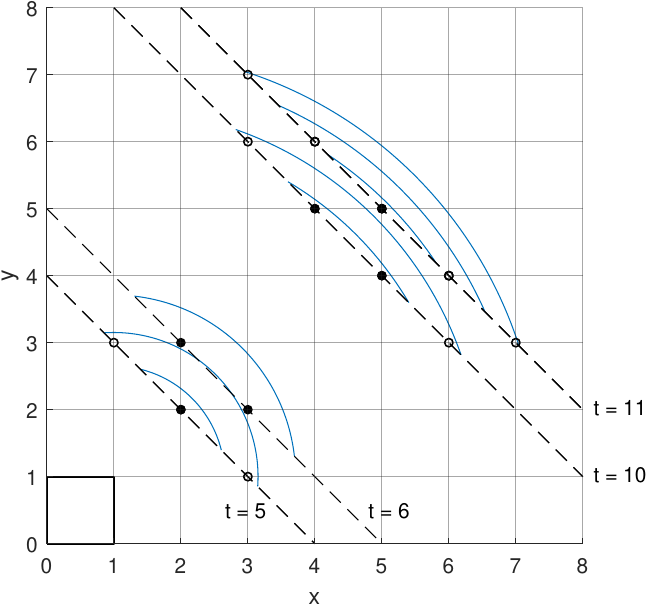}%
\caption{Tiles with lower-left corners $(i-1,j-1)$, $i+j-1=\tiles=\funt(\len)$ that can be reached from the reference tile on a unit square grid. Examples for several values of $\tiles$ and $\len$.
}%
\label{fig: probmax_terms}%
\end{figure}%

As can be seen in Figure~\ref{fig: probmax_terms}, for each $\tiles$ there is a maximum number of tiles that can be reached, beyond which incrementing $\len$ only results in $\tiles$ increasing by $1$. For a given $\len$, the number of reachable tiles can be determined by counting how many markers are covered by the stretch $w$ defined in Figure~\ref{fig: probmax_nterms}. This satisfies the relationship $w^2 = \len^2 - (\tiles-3)^2/2$. From the figure, the number of terms in the sum over $k$ (empty circular markers) is $\lceil w/\sqrt 2 - 1 \rceil$ for $\tiles$ odd and $\lceil w/\sqrt 2 - 3/2 \rceil$ for $\tiles$ even. This establishes \eqref{eq: probmax expl}.

\begin{figure}
\centering%
\subfigure[Odd $\tiles$]{%
\label{fig: probmax_nterms_odd}%
\includegraphics[width=.65\textwidth]{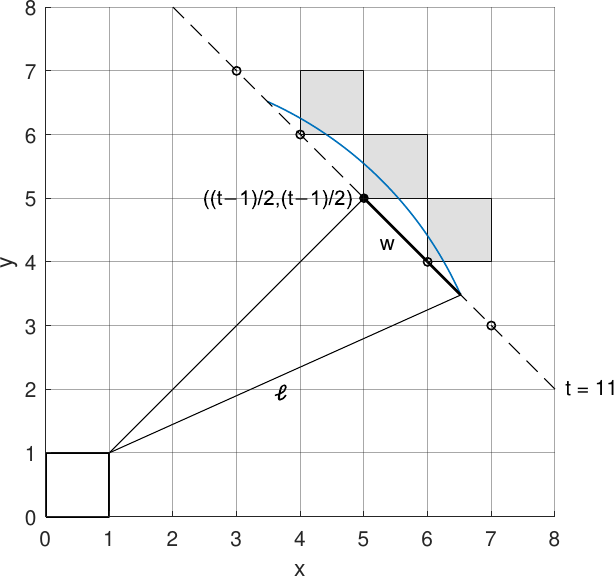}%
}\\%
\subfigure[Even $\tiles$]{%
\label{fig: probmax_nterms_even}%
\includegraphics[width=.65\textwidth]{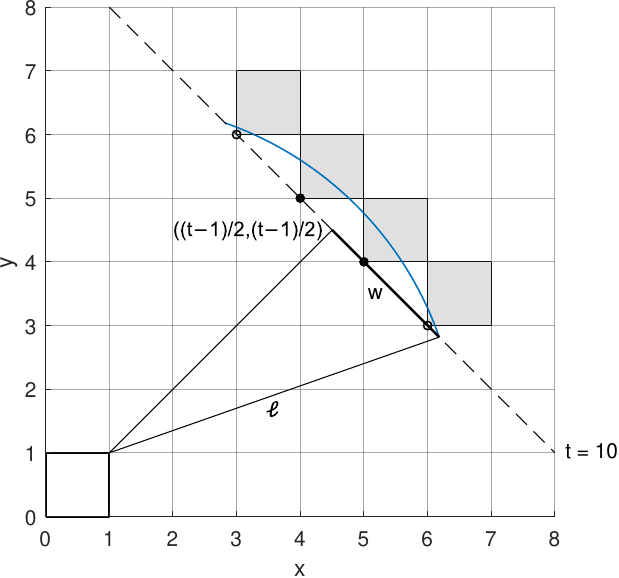}%
}%
\caption{Computation of the number of reachable tiles with lower-left corners $(i-1,j-1)$, $i+j-1=\tiles=\funt(\len)$ on a unit square grid}%
\label{fig: probmax_nterms}%
\end{figure}%

It should be noted that, in principle, the technique used in this proof could be employed for obtaining the probability that the number of tiles visited by the segment equals or exceeds any given value $\tiles < \funt(\len)$. However, the process is more cumbersome. Specifically, with reference to Figure~\ref{fig: probmax_funt_odd}, the arc in this more general setting can enclose any number of corners of the considered tile, not necessarily $1$, which makes the characterization of the shaded rectangles more complicated.
\end{proof}

The probability $\probmax(\len)$ of visiting the maximum number of tiles on a unit square grid, computed from Theorem~\ref{theo: probmax, sq}, is shown in Figure~\ref{fig: probmax_len}. As $\len$ grows, $\probmax(\len)$ has a jump when $\funt(\len)$ increases by $1$. This happens when $\len$ equals the right endpoint of the interval \eqref{eq: square, odd max tiles, lengths} for $\tiles$ odd or of the interval \eqref{eq: square, even max tiles, lengths} for $\tiles$ even. These length values will be denoted as $\len_\tiles$:
\begin{equation}
\label{eq: len tiles}
\len_\tiles = \begin{cases}
\displaystyle
\frac{\sqrt{(\tiles-3)^2 + (\tiles-1)^2}} {2} = \sqrt{\frac{(\tiles-2)^2+1}{2}} & \text{ for $\tiles$ odd, $\tiles \geq 3$} \\[2 mm] 
\displaystyle
\frac{\tiles-2}{\sqrt{2}} & \text{ for $\tiles$ even, $\tiles \geq 4$}.
\end{cases}
\end{equation}
Clearly, $\funt(\len_\tiles) = \tiles$ and $\lim_{\len \rightarrow \len_\tiles^- } \probmax(\len) = \probmax(\len_\tiles)$.

\begin{figure}%
\centering%
\includegraphics[width=.74\textwidth]{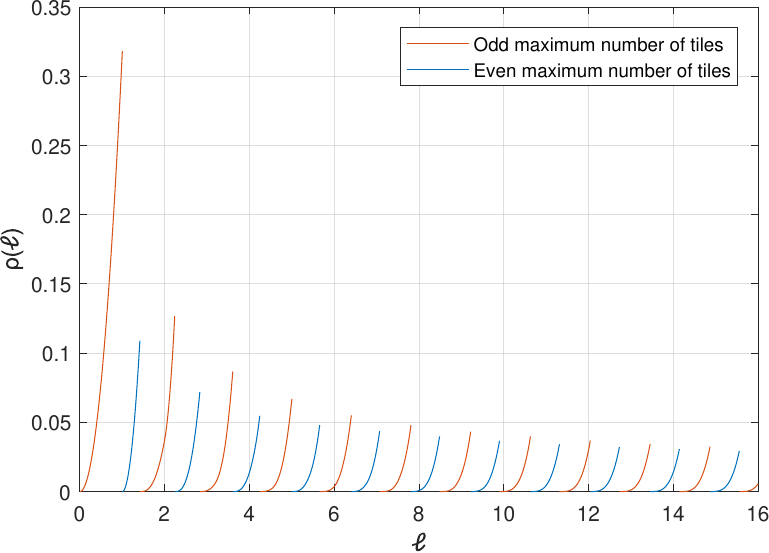}%
\caption{Probability $\probmax(\len)$ that a random segment of length $\len$ visits the maximum number of tiles on a unit square grid%
}%
\label{fig: probmax_len}%
\end{figure}%

The first continuous section seen in Figure~\ref{fig: probmax_len} corresponds to a maximum number of visited tiles $\tiles=3$, for lengths in the interval $(0, \len_3]$, where $\len_3= 1$. The second corresponds to $\tiles=4$, for lengths in $(\len_3, \len_4]$, where $\len_4= \sqrt{2}$. Within each continuous section the probability monotonically increases from $0$ to a maximum value. The heights of the maxima are asymptotically proportional to $1/\sqrt{\len_\tiles}$, as established by the next result.

\begin{proposition}
The probability $\probmax(\len)$ of visiting the maximum number of tiles on a unit square grid has the following asymptotic behaviour:
\begin{align}
\label{eq: probmax liminf}
\liminf_{\len \rightarrow \infty}\, \probmax(\len) & = 0, \\ 
\label{eq: probmax lim sqrt tiles}
\lim_{\tiles \rightarrow \infty}\, \sqrt{\tiles} \probmax(\len_\tiles) & = \frac {32\sqrt{2}}{105\pi}, \\
\label{eq: probmax limsup sqrt len}
\limsup_{\len \rightarrow \infty}\, \sqrt{\len} \probmax(\len) & = \frac {32\sqrt[4]{2}}{105\pi}.
\end{align}
\end{proposition}

\begin{proof}
The probability $\probmax(\len)$ can be made arbitrarily close to $0$ by choosing $\len$ close enough to the left endpoint of either \eqref{eq: square, odd max tiles, lengths} or \eqref{eq: square, even max tiles, lengths}. This proves \eqref{eq: probmax liminf}.

As argued in the proof of Theorem~\ref{theo: probmax, sq}, $\probmax(\len_\tiles)$ is the sum of several terms $\probmax_{i,j}(\len_\tiles) = g(\len_\tiles, i-2, j-2)$, where the sum runs over $i,j$ as given by \eqref{eq: probmax impl}. Each term is associated with one of the tiles shown in Figure~\ref{fig: probmax_nterms}, and includes a factor of $4$ to account for the symmetrical tiles in the other quadrants. For the purpose of establishing \eqref{eq: probmax lim sqrt tiles}, consider the graph in Figure~\ref{fig: probmax_nterms} with $\len= \len_\tiles$. It is convenient to rotate this graph by $45^\circ$ and shift it so that the original $x+y+1 = \tiles$ line (shown dashed in Figure~\ref{fig: probmax_nterms}) becomes coincident with the $x$ axis, and the arc center is at $(0,-(\tiles-3)/\sqrt{2})$. Figure~\ref{fig: probmax_asint_turned} shows the result, using $\tiles=9$ as an example. The circular markers in this figure correspond, in the graph before rotating, to the lower-left corners $(i-1,j-1)$ of the tiles with $i+j-1=\tiles$, $(i-2)^2+(j-1)^2<\len_\tiles$. Each such tile can be (partially) seen above its marker in Figure~\ref{fig: probmax_asint_turned}, its sides forming $45^\circ$ degrees with respect to the horizontal direction.

Given $\tiles$, let $N_\tiles$ be defined as the number of summands in \eqref{eq: probmax impl} for $\len=\len_\tiles$, and $i_\tiles^- = \min\{i \geq 2 \st (i-2)^2 + (\tiles-i-1)^2 < \len_\tiles^2\}$. For $n=1,\ldots,N_\tiles$, let
\begin{align}
\label{eq: i equiv n}
i(\tiles,n) &= i_\tiles^- + n -1, \\
\label{eq: j equiv n}
j(\tiles,n) &= \tiles+1-i(\tiles,n) = \tiles + 2 - i_\tiles^- - n.
\end{align}
Then the indices $i,j$ in \eqref{eq: probmax impl} for $\len=\len_\tiles$ can be replaced by the single index $n$:
\begin{equation}
\label{eq: probmax impl n}
\probmax(\len_\tiles) = \sum_{n=1}^{N_\tiles} \probmax_{i(\tiles,n), j(\tiles,n)}(\len_\tiles) = \sum_{n=1}^{N_\tiles} g\left(\len_\tiles, i(\tiles,n)-2, j(\tiles,n)-2 \right).
\end{equation}
The index $n = 1,\ldots,N_\tiles$  increases from left to right in Figure~\ref{fig: probmax_asint_turned}. Let $x_{\tiles, n}$ denote the $x$ coordinate of the $n$-th tile corner (circular marker) in this figure, and define
\begin{equation}
P_{\tiles,n} = \tiles \probmax_{i(\tiles,n),j(\tiles,n)}(\len_\tiles).
\end{equation}
Then, from \eqref{eq: probmax impl n},
\begin{equation}
\label{eq: sum p t n}
\tiles \probmax(\len_\tiles) = \sum_{n=1}^{N_\tiles} P_{\tiles,n}.
\end{equation}

\begin{figure}%
\centering
\includegraphics[width=.73\textwidth]{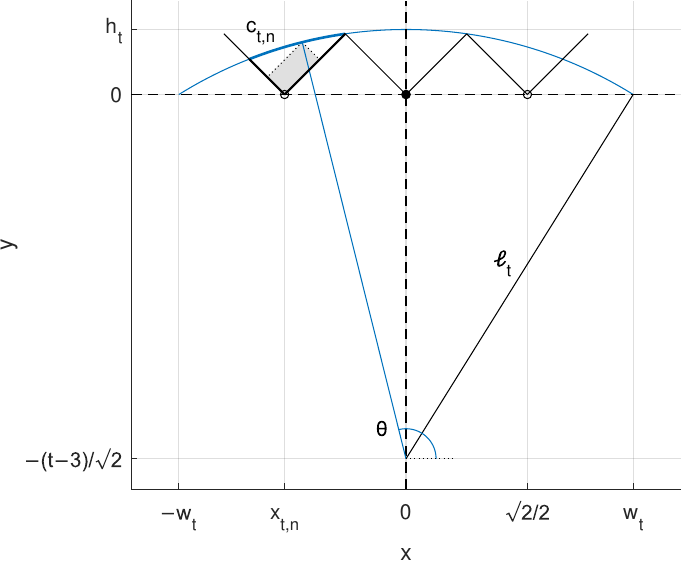}%
\caption{Reachable tiles that maximize the number of visited tiles on a unit square grid; rotated view. Example for $\tiles=9$}
\label{fig: probmax_asint_turned}%
\end{figure}%

Using the coordinate axes defined in Figure~\ref{fig: probmax_asint_turned}, the circular arc satisfies the equation $y = \funcirc(x)$ with
\begin{equation}
\label{eq: funcirc}
\funcirc(x) = h_\tiles - \len_\tiles + \sqrt{\len_\tiles^2-x^2},
\end{equation}
where the arc height $h_\tiles$ is 
\begin{equation}
\label{eq: h tiles}
h_\tiles = \len_\tiles - \frac{\tiles-3}{\sqrt{2}}.
\end{equation}
The $y$ coordinate of the arc for $x=x_{\tiles,n}$ will be denoted as $y_{\tiles,n} = \funcirc(x_{\tiles,n})$. The half-width of the arc, $w_\tiles$, is computed in the proof of Theorem~\ref{theo: probmax, sq} (or from the right triangle in Figure~\ref{fig: probmax_asint_turned}), and using \eqref{eq: len tiles} it can be expressed as
\begin{equation}
\label{eq: w}
w_\tiles = \sqrt{\len_\tiles^2-\frac{(\tiles-3)^2}{2}} =
\begin{cases}
\displaystyle
\sqrt{\tiles-2} & \text{ for $\tiles$ odd} \\
\displaystyle
\sqrt{\tiles-5/2} & \text{ for $\tiles$ even}.
\end{cases}
\end{equation}
The inequality
\begin{equation}
\label{eq: sqrt asympt}
\sqrt{1+\genvar} < 1 + \frac \genvar 2 \quad\text{for } s \geq -1, \ s \neq 0
\end{equation}
implies that
\begin{equation}
\label{eq: asympt w}
w_\tiles \leq \sqrt{\tiles-2} < \sqrt{\tiles} - 1/\sqrt\tiles.
\end{equation}
Similarly, \eqref{eq: len tiles} and \eqref{eq: sqrt asympt} yield, considering that $\tiles \geq 3$,
\begin{equation}
\label{eq: asympt len tiles}
\frac{\tiles-2}{\sqrt{2}} \leq \len_\tiles < \frac{\tiles-2}{\sqrt{2}} + \frac{1}{2\sqrt{2}(\tiles-2)}
\leq \frac{\tiles-3/2}{\sqrt{2}}.
\end{equation}
The arc height $h_\tiles$ is bounded from \eqref{eq: h tiles} and \eqref{eq: asympt len tiles} as
\begin{equation}
\label{eq: h bound}
\frac {\sqrt{2}} 2 \leq h_\tiles <
\frac{3\sqrt{2}}{4}.
\end{equation}
The number $N_\tiles$ of reachable tiles shown in the figure (circular markers) is easily obtained from $w_\tiles$ as
\begin{equation}
\label{eq: N tiles}
N_\tiles = 
\begin{cases}
\displaystyle
2\left\lceil \frac{w_\tiles}{\sqrt{2}} \right\rceil - 1 = 2\left\lceil \sqrt{\frac{\tiles-2}{2}} \right\rceil - 1 & \text{ for $\tiles$ odd} \\[5mm] 
\displaystyle
2\left\lceil \frac{w_\tiles}{\sqrt{2}} - \frac 1 2 \right\rceil = 2\left\lceil \sqrt{\frac{\tiles-5/2}{2}} - \frac 1 2 \right\rceil & \text{ for $\tiles$ even}.
\end{cases}
\end{equation}

The asymptotic analysis to be presented is based on the following observations. The half-width of the arc $w_\tiles$ grows asymptotically as $\sqrt{\tiles}$ for $\tiles \rightarrow \infty$, whereas the radius $\len_\tiles$ is asymptotically proportional to  $\tiles$, and the height $h_\tiles$ converges to $\sqrt{2}/2$. As a result, the slope $\diff \funcirc(x)/\diff x$ tends to $0$.

Consider the tile depicted with thick line in Figure~\ref{fig: probmax_asint_turned}, for a generic value of the index $n$. For this tile, $P_{\tiles,n}/t = \probmax_{i(\tiles,n),j(\tiles,n)}(\len_\tiles)$ can be obtained as the area of the shaded rectangle in the figure averaged over $\theta$ uniformly distributed on $[0,2\pi)$, and then multiplied by $4$, in the same way as in the proof of Theorem~\ref{theo: probmax, sq}. Equivalently, the average can be restricted to the range of values of $\theta$ for which the shaded area is non-zero, corresponding to the arc with thick line in the figure, of length $c_{\tiles,n}$, and the result multiplied by $c_{\tiles,n}/(2\pi\len_\tiles)$. Denoting the average over this restricted range by $A_{\tiles,n}$,
\begin{equation}
\label{eq: P i j asympt 1}
P_{\tiles,n} = \frac{2 \tiles A_{\tiles,n} c_{\tiles,n}}{\pi\len_\tiles}.
\end{equation}
Taking into account that $A_{\tiles,n}$ and $c_{\tiles,n}$ are upper-bounded by the area of a tile and one quarter of the length of a circle of radius $1$, respectively,
\begin{align}
\label{eq: A tiles n bound}
0 &< A_{\tiles,n} < 1, \\
\label{eq: c tiles n bound}
0 &< c_{\tiles,n} \leq \frac \pi 2,
\end{align}
and using \eqref{eq: asympt len tiles},
\begin{align}
\label{eq: P i j asympt 2 1}
P_{\tiles,n}
&\leq \left(1 + \frac 2 {\tiles-2} \right) \frac{2 \sqrt{2} A_{\tiles,n} c_{\tiles,n}}{\pi} < \frac{2\sqrt{2} A_{\tiles,n} c_{\tiles,n}}{\pi} + \frac{2\sqrt{2}}{\tiles-2}, \\
\label{eq: P i j asympt 2 2}
P_{\tiles,n}
&\geq \left(1 + \frac {3/2} {\tiles-3/2} \right) \frac{2 \sqrt{2} A_{\tiles,n} c_{\tiles,n}}{\pi} > \frac{2 \sqrt{2} A_{\tiles,n} c_{\tiles,n}}{\pi}.
\end{align}

Instead of calculating  $A_{\tiles,n}$ and $c_{\tiles,n}$ directly, it is easier to compute approximate versions of these, based on replacing the arc by its chord, as depicted in Figure~\ref{fig: probmax_asint_detail}. The difference between the approximate and exact values for each variable will be bounded from below and from above by functions of $\tiles$ that tend to $0$ as $\tiles \rightarrow \infty$. Regarding $c_{\tiles,n}$, the chord length, denoted as $\tilde c_{\tiles,n}$, is obtained from Figure~\ref{fig: probmax_asint_turned} as
\begin{equation}
\tilde c_{\tiles,n} =  2\len_\tiles \sin\frac{c_{\tiles,n}}{2\len_\tiles}. 
\end{equation}
Taking into account that $\sin\genvar > \genvar - {\genvar^3}/6$ for $s > 0$ and using \eqref{eq: asympt len tiles},
\begin{equation}
\label{eq: tilde c bound diff}
0 > \tilde c_{\tiles,n} - c_{\tiles,n} > -\frac{c_{\tiles,n}^3}{24\len_\tiles^2} \geq -\frac{\pi^3}{192\len_\tiles^2} \geq -\frac {\pi^3} {96(\tiles-2)^2}.
\end{equation}

\begin{figure}
\centering
\includegraphics[width=.82\textwidth]{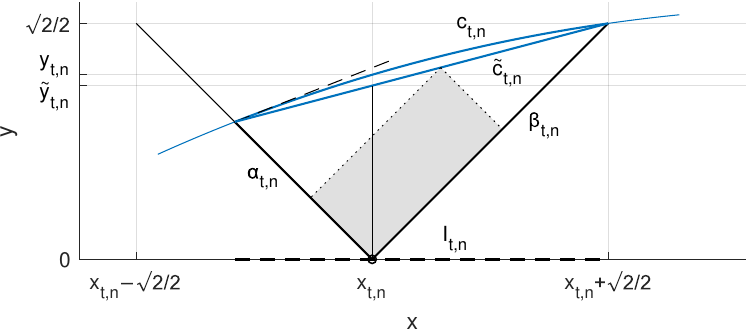}%
\caption{Computation of $P_{\tiles,n}$}%
\label{fig: probmax_asint_detail}%
\end{figure}%

As for $A_{\tiles,n}$, let $\tilde A_{\tiles,n}$ denote its approximate version where the upper corner of the shaded rectangle in Figure~\ref{fig: probmax_asint_detail} is uniformly distributed on the chord rather than on the arc. This approximation entails two types of errors: the distribution of the angle $\theta$ is no longer uniform; and the shaded rectangle becomes smaller because its corner is on the chord, not on the arc. It is convenient to describe the two sources of error in terms of, respectively, the horizontal and vertical coordinates of the upper corner of the shaded rectangle. More specifically, let $x$ and $y$ denote the coordinates of the corner along the arc, and let $\intervx_{\tiles,n}$ denote the interval of values of $x$ corresponding to the projection of the arc (or of the chord) onto the horizontal axis. This interval is shown in Figure~\ref{fig: probmax_asint_detail} with dashed, thick line. Firstly, in the exact case the distribution of $x$ is non-uniform on $\intervx_{\tiles,n}$ (corresponding to a uniform distribution of $\theta$), whereas the approximation implies that $x$ is uniformly distributed on $\intervx_{\tiles,n}$. The corresponding probability density functions will be respectively denoted as $\fdp_{\tiles,n}(x)$ and $\tilde \fdp_{\tiles,n}(x)$. Secondly, for a given $x$, the vertical coordinate in the exact case, $y$ (determined by the arc), is replaced by a smaller value $\tilde y$ in the approximation (chord). Let $R_{\tiles,n}(x)$ and $\tilde R_{\tiles,n}(x)$ denote the area of the shaded rectangle that has its upper corner on the arc and on the chord respectively, as a function of $x$. Then
\begin{align}
\label{eq: A int R fdp x}
A_{\tiles,n} &= \int_{\intervx_{\tiles,n}} R_{\tiles,n}(x) \fdp_{\tiles,n}(x) \, \diff x, \\
\label{eq: tilde A int R fdp x}
\tilde A_{\tiles,n} &= \int_{\intervx_{\tiles,n}} \tilde R_{\tiles,n}(x) \tilde \fdp_{\tiles,n}(x) \, \diff x.
\end{align}

Consider the first type of error in the approximation of $A_{\tiles,n}$, which arises from using $\tilde \fdp_{\tiles,n}(x)$ instead of $\fdp_{\tiles,n}(x)$. In the following it will be assumed that $\tiles \geq 4$. It stems from Figure~\ref{fig: probmax_asint_turned} that any point $(x,y)$ on the full arc defined by \eqref{eq: funcirc} satisfies the equality $\cot \theta = \sqrt{2} x /(\tiles-3)$. Therefore
\begin{equation}
\label{eq: d x d theta}
\frac{\diff x}{\diff \theta} = -\frac{\tiles-3}{\sqrt{2} \sin^2 \theta}.
\end{equation}
Along this arc, with $x \in [-w_\tiles, w_\tiles]$, the minimum of $|\sin \theta|$ is attained when $x$ equals $w_\tiles$ or $-w_\tiles$. Thus, using \eqref{eq: asympt len tiles},
\begin{equation}
\label{eq: sin theta bound}
1 \geq |\sin \theta| \geq \frac{\tiles-3}{\sqrt{2}\len_\tiles} > \frac{\tiles-3}{\tiles-3/2}.
\end{equation}
In the exact case $\theta$ is uniformly distributed along the arc shown with thick line in Figure~\ref{fig: probmax_asint_turned}, and the probability density function of $x$ can be obtained using the transformation theorem for continuous random variables \cite[section~5.2]{Papoulis91}:
\begin{equation}
\label{eq: fdp x}
\fdp_{\tiles,n}(x) = \frac {\len_\tiles} {c_{\tiles,n} |\diff x / \diff \theta|} \quad\text{for } x \in \intervx_{\tiles,n}.
\end{equation}
It follows from \eqref{eq: d x d theta}--\eqref{eq: fdp x} that
\begin{equation}
\label{eq: max min tilde pdf}
\frac {\max_{x \in \intervx_{\tiles,n}} \fdp_{\tiles,n}(x)} {\min_{x \in \intervx_{\tiles,n}} \fdp_{\tiles,n}(x)} < \left(\frac{\tiles-3/2}{\tiles-3}\right)^2.
\end{equation}
Since $\fdp_{\tiles,n}(x)$ and $\tilde\fdp_{\tiles,n}(x)$ are probability density functions, it cannot be the case that one of them is greater than the other for all values of $x$. In addition, both functions have the same support, namely $\intervx_{\tiles,n}$, on which $\tilde \fdp_{\tiles,n}(x)$ is constant. This implies that, for $x \in \intervx_{\tiles,n}$ and $\tiles \geq 4$,
\begin{align}
\label{eq: tilde pdf pdf bound >}
\frac {\tilde \fdp_{\tiles,n}(x)} {\fdp_{\tiles,n}(x)} &> \left(\frac{\tiles-3}{\tiles-3/2}\right)^2 = \left(1-\frac{3/2}{\tiles-3/2}\right)^2 > 1 - \frac 3 {\tiles-3/2}, \\
\label{eq: tilde pdf pdf bound <}
\frac {\tilde \fdp_{\tiles,n}(x)} {\fdp_{\tiles,n}(x)} &< \left(\frac{\tiles-3/2}{\tiles-3}\right)^2 = \left(1+\frac{3/2}{\tiles-3}\right)^2 \leq 1 + \frac {21/4} {\tiles-3}.
\end{align}

To analyze the second type of error in the approximation of $A_{\tiles,n}$, caused by using $\tilde R_{\tiles,n}(x)$ instead of $R_{\tiles,n}(x)$, consider the tangent line displayed in Figure~\ref{fig: probmax_asint_detail} (dashed line). This makes it easy to see that for a given $x$ the difference between the vertical coordinate of the arc, $y$, and that of the chord, $\tilde y$, is less than the chord length, $\tilde c_{\tiles,n}$, multiplied by the maximum of the absolute value of the arc slope. The former is at most $\sqrt{2}$. The latter is upper-bounded by the value of $|\diff \funcirc(x)/\diff x|$ at $x=w_\tiles$ or $x=-w_\tiles$. From \eqref{eq: funcirc}, \eqref{eq: w} and \eqref{eq: sqrt asympt}, making use of the assumption that $\tiles \geq 4$,
\begin{equation}
\label{eq: asympt slope bound}
\begin{split}
\left| \frac{\diff \funcirc(x)}{\diff x} \right| &= \frac{|x|}{\sqrt{\len_\tiles^2-x^2}} \leq \frac{w_\tiles}{\sqrt{\len_\tiles^2-w_\tiles^2}} \leq \frac{\sqrt{2(\tiles-2)}}{\tiles-3} \\
& < \sqrt{\frac 2 {\tiles-3}} \left(1 + \frac 1 {2(\tiles-3)} \right) \leq \frac 3 2 \sqrt{\frac 2 {\tiles-3}},
\end{split}
\end{equation}
and thus
\begin{equation}
\label{eq: y: arc - chord}
0 \leq y - \tilde y < \frac 3 {\sqrt{\tiles-3}}.
\end{equation}
This reduction from $y$ to $\tilde y$ translates into a difference smaller than $3 / \sqrt{2(\tiles-3)}$ in each dimension of the shaded rectangle; or, since those dimensions are at most $1$, a difference in area smaller than twice that value:.
\begin{equation}
\label{eq: R: arc - chord}
0 \leq R_{\tiles,n}(x) - \tilde R_{\tiles,n}(x) < \frac{3\sqrt{2}} {\sqrt{\tiles-3}}.
\end{equation}

Combining \eqref{eq: tilde pdf pdf bound >} and \eqref{eq: R: arc - chord} with \eqref{eq: tilde A int R fdp x}, and using \eqref{eq: A tiles n bound} and \eqref{eq: A int R fdp x}, $\tilde A_{\tiles,n} - A_{\tiles,n}$ can be bounded for $\tiles \geq 4$ as
\begin{equation}
\label{eq: tilde A bound diff 1}
\begin{split}
\tilde A_{\tiles,n} - A_{\tiles,n} &= \int_{\intervx_{\tiles,n}} \tilde R_{\tiles,n}(x) \tilde \fdp_{\tiles,n}(x) \, \diff x - A_{\tiles,n} \\
&> \left( 1 - \frac 3 {\tiles-3/2} \right) \bigintssss_{\intervx_{\tiles,n}} \biggl( R_{\tiles,n}(x) - \frac{3\sqrt{2}} {\sqrt{\tiles-3}} \biggr) \fdp_{\tiles,n}(x) \, \diff x - A_{\tiles,n} \\
&> \left( 1 - \frac 3 {\tiles-3/2} \right) \biggl( A_{\tiles,n} - \frac{3\sqrt{2}}{\sqrt{\tiles-3}} \biggr) - A_{\tiles,n} > -\frac 3 {\tiles-3/2} - \frac{3\sqrt{2}}{\sqrt{\tiles-3}}.
\end{split}
\end{equation}
Likewise, with \eqref{eq: tilde pdf pdf bound <} instead of \eqref{eq: tilde pdf pdf bound >},
\begin{equation}
\label{eq: tilde A bound diff 2}
\tilde A_{\tiles,n} - A_{\tiles,n} \leq \left( 1 + \frac {21/4} {\tiles-3} \right) \int_{\intervx_{\tiles,n}} R_{\tiles,n}(x) \fdp_{\tiles,n}(x) \, \diff x - A_{\tiles,n} = \frac {21/4} {\tiles-3}\, A_{\tiles,n} < \frac {21/4} {\tiles-3}.
\end{equation}

The difference between $\tilde A_{\tiles,n} \tilde c_{\tiles,n}$ and $A_{\tiles,n} c_{\tiles,n}$ is bounded as follows:
\begin{equation}
\begin{split}
\label{eq: tilde A c bound diff 1}
|\tilde A_{\tiles,n} \tilde c_{\tiles,n} - A_{\tiles,n} c_{\tiles,n}|
&= |A_{\tiles,n} (\tilde c_{\tiles,n} - c_{\tiles,n}) + (\tilde A_{\tiles,n} - A_{\tiles,n}) \tilde c_{\tiles,n}| \\
&\leq A_{\tiles,n} \left|\tilde c_{\tiles,n} - c_{\tiles,n} \right| + \tilde c_{\tiles,n} |\tilde A_{\tiles,n} - A_{\tiles,n}|.
\end{split}
\end{equation}
Using \eqref{eq: A tiles n bound}, \eqref{eq: c tiles n bound}, \eqref{eq: tilde c bound diff}, \eqref{eq: tilde A bound diff 1} and \eqref{eq: tilde A bound diff 2} into \eqref{eq: tilde A c bound diff 1},
\begin{equation}
\begin{split}
\label{eq: tilde A c bound diff 2}
|\tilde A_{\tiles,n} \tilde c_{\tiles,n} - A_{\tiles,n} c_{\tiles,n}|
&< \frac {\pi^3} {96(\tiles-2)^2} + \frac{\pi}{2} \max\left\{
\frac 3 {\tiles-3/2} + \frac{3\sqrt{2}}{\sqrt{\tiles-3}}, \frac {21/4} {\tiles-3}
\right\}.
\end{split}
\end{equation}
From \eqref{eq: P i j asympt 2 1}, \eqref{eq: P i j asympt 2 2} and \eqref{eq: tilde A c bound diff 2} it stems that
\begin{equation}
\label{eq: P i j asympt 3 bis}
P_{\tiles,n} = \frac{2\sqrt{2} \tilde A_{\tiles,n} \tilde c_{\tiles,n}}{\pi} + E_{\tiles, n}
\end{equation}
with $|E_{\tiles, n}| < e(\tiles)$, where the function $e$ depends only on $\tiles$ and tends to $0$ as $\tiles \rightarrow 0$.

Let $\tilde y_{\tiles,n}$ denote the $y$ coordinate of the chord for $x = x_{\tiles,n}$, as shown in Figure~\ref{fig: probmax_asint_detail}, and let $m_{\tiles,n}$ denote the slope of the chord. By construction $|m_{\tiles,n}|<1$. Consider the lengths $\lra_{\tiles,n}$ and $\lrb_{\tiles,n}$ defined in the figure. Then
\begin{equation}
\frac{\tilde y_{\tiles,n}-\lra_{\tiles,n}/\sqrt{2}}{\lra_{\tiles,n}/\sqrt{2}} = m_{\tiles,n},
\end{equation}
which gives $\lra_{\tiles,n} = \sqrt{2}\tilde y_{\tiles,n}/(1+m_{\tiles,n})$. Likewise, $\lrb_{\tiles,n} = \sqrt{2}\tilde y_{\tiles,n}/(1-m_{\tiles,n})$. With the upper corner of the shaded rectangle uniformly distributed along the chord, its average area $\tilde A_{\tiles,n}$ is easily computed in terms of $\lra_{\tiles,n}$ and $\lrb_{\tiles,n}$, and therefore of $\tilde y_{\tiles,n}$ and $m_{\tiles,n}$:
\begin{equation}
\label{eq: tilde A y}
\tilde A_{\tiles,n} = \lra_{\tiles,n}\lrb_{\tiles,n} \int_0^1 \eta(1-\eta)\, \diff \eta = \frac{\lra_{\tiles,n}\lrb_{\tiles,n}} 6 = \frac{\tilde y_{\tiles,n}^2}{3(1-m_{\tiles,n}^2)}.
\end{equation}
Similarly,
\begin{equation}
\label{eq: tilde c y}
\tilde c_{\tiles,n} = \sqrt{\lra_{\tiles,n}^2+\lrb_{\tiles,n}^2} = \frac {2 \tilde y_{\tiles,n} \sqrt{1+m_{\tiles,n}^2}} {1-m_{\tiles,n}^2}.
\end{equation}
Substituting \eqref{eq: tilde A y} and \eqref{eq: tilde c y} into \eqref{eq: P i j asympt 3 bis},
\begin{equation}
\label{eq: P i j asympt 4}
P_{\tiles,n}
= \frac{4\sqrt{2} \sqrt{1+m_{\tiles,n}^2}} {3\pi(1-m_{\tiles,n}^2)^2}\, \tilde y_{\tiles,n}^3 + E_{\tiles, n}.
\end{equation}
The inequality \eqref{eq: asympt slope bound} applies in particular to $|m_{\tiles,n}|$, and gives $m_{\tiles,n}^2 < 9/(2(\tiles-3))$. For simplicity, assume $\tiles \geq 8$ so that this bound is less than $1$. Then, making use of \eqref{eq: sqrt asympt}, the term depending on $m_{\tiles,n}^2$ in \eqref{eq: P i j asympt 4} satisfies
\begin{equation}
\label{eq: P i j asympt 4 bis}
\frac{\sqrt{1+m_{\tiles, n}^2}} {(1-m_{\tiles, n}^2)^2} \leq \frac{(\tiles-3)^2}{(\tiles-15/2)^2} \left(1+\frac {9/4} {\tiles-3}\right) = 1 + \frac{9(5\tiles - 24)}{(2\tiles-15)^2}.
\end{equation}
From \eqref{eq: h bound},
\begin{equation}
\label{eq: P i j asympt 4 ter}
\tilde y_{\tiles,n} < y_{\tiles,n} \leq h_\tiles \leq \frac{3\sqrt{2}} 4.
\end{equation}
Using \eqref{eq: P i j asympt 4 bis} into \eqref{eq: P i j asympt 4} and taking into account \eqref{eq: P i j asympt 4 ter},
\begin{equation}
\label{eq: P i j asympt 5}
P_{\tiles,n} = \frac{4\sqrt{2}} {3\pi} \tilde y_{\tiles,n}^3 + E_{\tiles,n} + E'_{\tiles,n}
\end{equation}
with $|E'_{\tiles,n}| < e'(t)$, where $e'(t)$ tends to $0$ as $t \rightarrow \infty$.

Let $\funcircesc_\tiles$ be defined as a horizontally scaled version of $\funcirc$, where the variable $x \in [-w_\tiles, w_\tiles]$ is replaced by $\zeta = x/w_\tiles \in [-1,1]$:
\begin{equation}
\label{eq: funcircest tiles 1}
\funcircesc_\tiles(\zeta) = \funcirc\left(w_\tiles \zeta\right) = h_\tiles - \len_\tiles + \sqrt{\len_\tiles^2-w^2_\tiles\zeta^2}.
\end{equation}
As $\tiles \rightarrow \infty$, the sequence $\funcircesc_\tiles$ converges pointwise to the function
\begin{equation}
\label{eq: funcircesc}
\funcircesc(\zeta) = \lim_{\tiles \rightarrow \infty} \funcircesc_\tiles(\zeta) = \frac{1 -\zeta^2}{\sqrt{2}}.
\end{equation}
Furthermore, the convergence is uniform. To show this, it is easier to analyze the subsequences with $\tiles$ odd and even separately. For $\tiles$ odd, substituting $\len_\tiles$, $h_\tiles$  and $w_\tiles$ given by \eqref{eq: len tiles}, \eqref{eq: h tiles} and \eqref{eq: w} into \eqref{eq: funcircest tiles 1},
\begin{equation}
\label{eq: funcircesc, odd 1}
\funcircesc_\tiles(\zeta) =
-\frac{\tiles-3}{\sqrt{2}} + \sqrt{\frac{(\tiles-2)^2 - 2(\tiles-2)\zeta^2 +1 }{2}}.
\end{equation}
Computing $\partial \funcircesc_\tiles(\zeta)/\partial\tiles$ as if $\tiles$ were a continuous variable, and taking into account that $\zeta \in [-1,1]$,
\begin{equation}
\label{eq: der funcircesc, odd 1}
\frac{\partial \funcircesc_\tiles(\zeta)}{\partial \tiles} = \frac 1 {\sqrt{2}} \left(
-1 + \frac{\tiles-2-\zeta^2}{\sqrt{(\tiles-2)^2 - 2(\tiles-2)\zeta^2 + 1}} \right) \leq 0.
\end{equation}
This implies that the subsequence $\funcircesc_\tiles$ with $\tiles$ odd is monotone non-increasing. Since the limit function $\funcircesc$ is continuous, by Dini's theorem \cite[section~1.1]{Hirsch99} the convergence of this subsequence is uniform. An analogous argument establishes the uniform convergence of the subsequence for $\tiles$ even.
Therefore, the full sequence $\funcircesc_\tiles$ converges to $\funcircesc$ uniformly on $[-1,1]$.

The inequality \eqref{eq: y: arc - chord} applies, for $\tiles \geq 4$, to $y_{\tiles,n} = \funcircesc_\tiles(x_{\tiles,n}/w_\tiles)$ and $\tilde y_{\tiles,n}$. Together with \eqref{eq: P i j asympt 4 ter}, and making use of the uniform convergence of $\funcircesc_\tiles$ to $\funcircesc$, 
this implies that
\begin{equation}
\label{eq: funcircest tiles 5}
\tilde y_{\tiles,n}^3 = \funcircesc^3\left( \frac{x_{\tiles,n}} {w_\tiles} \right) + E''_{\tiles,n}
\end{equation}
with $|E''_{\tiles,n}| < e''(t)$, where $e''(\tiles)$ tends to $0$ as $\tiles \rightarrow \infty$.

From \eqref{eq: sum p t n}, \eqref{eq: P i j asympt 5} and \eqref{eq: funcircest tiles 5},
\begin{equation}
\label{eq: sum p t n aprox 1}
\tiles \probmax(\len_\tiles) = \frac{4 \sqrt{2}} {3\pi} \sum_{n=1}^{N_\tiles} \funcircesc^3 \left( \frac{x_{\tiles,n}}{w_\tiles} \right) + \sum_{n=1}^{N_\tiles} \left(E_{\tiles,n} + E'_{\tiles,n} + \frac{4 \sqrt{2}} {3\pi} E''_{\tiles,n} \right),
\end{equation}
which can be rewritten as
\begin{equation}
\label{eq: sum p t n aprox 2}
\sqrt{\tiles} \probmax(\len_\tiles) = \frac{4 w_\tiles} {3\pi\sqrt{\tiles}} \frac {\sqrt{2}} {w_\tiles} \sum_{n=1}^{N_\tiles} \funcircesc^3 \left( \frac{x_{\tiles,n}}{w_\tiles} \right)
+ \frac{{N_\tiles}} {\sqrt{t}} \frac 1 {N_\tiles} \sum_{n=1}^{N_\tiles} \left(E_{\tiles,n} + E'_{\tiles,n} + \frac{4 \sqrt{2}} {3\pi} E''_{\tiles,n} \right).
\end{equation}
It is clear from Figure~\ref{fig: probmax_asint_turned} that
\begin{alignat}{2}
-w_\tiles &< x_{\tiles,1} && \leq -w_\tiles + \sqrt{2}, \\
w_\tiles - \sqrt{2} &\leq x_{\tiles,N_\tiles} && < w_\tiles.
\end{alignat}
Thus the values $x_{\tiles,n}/w_\tiles$, $n=1,\ldots,N_\tiles$ are equispaced with step $\sqrt{2}/w_\tiles$. In addition, $x_{\tiles,1} / w_\tiles$ and $x_{\tiles,N_\tiles} / w_\tiles$ tend to $-1$ and $1$ respectively as $\tiles \rightarrow \infty$. In consequence, the term
\[
\frac{\sqrt{2}}{w_\tiles} \sum_{n=1}^{N_\tiles} \funcircesc^3 \left( \frac{x_{\tiles,n}}{w_\tiles} \right)
\]
in \eqref{eq: sum p t n aprox 2} can be interpreted as a Riemann sum that approximates the integral of $\funcircesc^3(\zeta)$ over $[-1,1]$. Since this function is continuous the sum indeed converges to the integral \cite[section~7.2]{Abbott15}. On the other hand,
\begin{equation}
\left| \frac 1 {N_\tiles} \sum_{n=1}^{N_\tiles} \left(E_{\tiles,n} + E'_{\tiles,n} + \frac{4 \sqrt{2}} {3\pi} E''_{\tiles,n} \right) \right| < e(\tiles)+e'(\tiles) + \frac{4 \sqrt{2}} {3\pi} e''(\tiles)
\end{equation}
and the right-hand side tends to $0$ as $\tiles \rightarrow \infty$. The terms $w_\tiles/\sqrt{\tiles}$ and $N_\tiles/\sqrt{\tiles}$ in \eqref{eq: sum p t n aprox 2} tend to $1$ and $\sqrt{2}$ respectively, according to \eqref{eq: w} and \eqref{eq: N tiles}. Thus, taking limits and substituting \eqref{eq: funcircesc},
\begin{equation}
\label{eq: sum p t n lim}
\lim_{\tiles \rightarrow \infty} \sqrt{\tiles} \probmax(\len_\tiles) =
\frac{4} {3\pi} \int_{-1}^1 \funcircesc^3(\zeta) \,\diff \zeta = \frac{2\sqrt{2}} {3\pi} \int_{0}^1 (1-\zeta^2)^3 \,\diff \zeta = \frac{32\sqrt{2}} {105\pi},
\end{equation}
which establishes \eqref{eq: probmax lim sqrt tiles}.

Lastly, \eqref{eq: probmax limsup sqrt len} is obtained from \eqref{eq: probmax lim sqrt tiles} by observing that $\limsup_{\len \rightarrow \infty} \sqrt{\len} \probmax(\len) = \lim_{\tiles \rightarrow \infty} \sqrt{\len_\tiles} \probmax(\len_\tiles)$ and that $ \lim_{\tiles \rightarrow \infty} \len_\tiles/\tiles = 1 / \sqrt{2}$.
\end{proof}


\providecommand{\noopsort}[1]{}

\end{document}